\documentclass{amsart}
\usepackage[margin=1in]{geometry}

\usepackage{hyperref}
\usepackage[numbers]{natbib}

\usepackage{amsaddr}
\usepackage{amsmath,amssymb,amsthm}
\usepackage{bm} 
\usepackage{booktabs} 
\usepackage{mathtools} 
\usepackage{stmaryrd} 
\usepackage{mdframed} 
\usepackage{multirow} 
\usepackage{pbox} 

\def\tr{\intercal} 
\newcommand{\sqrtb}[1]{\left(#1\right)^{1/2}}

\def\R{\mathbb{R}}
\def\Ref{\mathcal{R}}
\def\Q{\mathcal{Q}}
\def\T{\mathcal{T}}
\def\O{\mathcal{O}}

\def\nc{n_c} 

\def\oned{\mathrm{1D}}
\def\Fhat{\widehat{\bm F}}

\def\uu{\mathtt{u}}
\def\G{\mathtt{G}}
\def\J{\mathtt{J}}
\def\Jinv{\mathtt{GJ}^{-1}}
\def\Jinvidp{\widehat{\mathtt{GJ}}^{-1}}
\def\FF{\mathtt{F}}
\def\FFhat{\widehat{\mathtt{F}}}
\def\M{\mathtt{M}}
\def\D{\mathtt{D}}
\def\Didp{\widehat{\D}}
\def\B{\mathtt{B}}
\def\n{\mathtt{n}}
\def\wn{\mathtt{wn}}

\newtheorem{prop}{Proposition}
\newtheorem{defn}{Definition}
\newtheorem*{rem*}{Remark}

\title%
[Sparse invariant domain preserving DG methods with convex limiting]%
{Sparse invariant domain preserving discontinuous Galerkin methods with subcell convex limiting}
\author{Will Pazner}
\address{Center for Applied Scientific Computing, Lawrence Livermore National Laboratory}

\begin{document}

\maketitle

\begin{abstract}
In this paper, we develop high-order nodal discontinuous Galerkin methods for hyperbolic conservation laws that satisfy invariant domain preserving properties using a subcell flux corrections and convex limiting.
These methods are based on a subcell flux corrected transport (FCT) methodology, that involves blending a high-order target scheme with a robust, low-order invariant domain preserving method that is obtained using a graph viscosity technique.
The new low-order discretizations are based on sparse stencils which do not increase with the polynomial degree of the high-order DG method.
As a result, the accuracy of the low-order method does not degrade when used with high-order target methods.
The method is applied to both scalar conservation laws, for which the discrete maximum principle is naturally enforced, and to systems of conservation laws such as the Euler equations, for which positivity of density and a minimum principle for specific entropy are enforced.
Numerical results are presented on a number of benchmark test cases.
\end{abstract}

\section{Introduction}

High-order numerical methods have been successfully applied to a wide range of applications \cite{Shu2016,Slotnick2014,Wang2013}.
These methods promise higher accuracy per degree of freedom when compared with traditional low-order methods, and have the potential to achieve high efficiency on modern computing architectures \cite{Hutchinson2016,Brown2018,Franco2019a}.
For example, in the field of computational fluid dynamics, such methods have seen particular success when applied to under-resolved turbulent flows, such as in the context of implicit large eddy simulation (ILES) \cite{Moura2017,Pazner2017,Carton-de-Wiart2015}.
However, a critical issue that must be addressed when applying high-order methods to convection-dominated problems is their robustness, especially in the context of nonlinear problems with discontinuous features such as shock waves \cite{Moura2017a,Zahr2018,Pazner2019}.

In particular, discontinuous Galerkin (DG) methods have seen considerable success when applied to convection dominated problems \cite{Cockburn2001}.
These methods possess many desirable properties, such as arbitrary formal order of accuracy, and suitability for use with unstructured meshes.
The robustness of DG methods is the subject of a large body of research \cite{Klose2019,Moura2017}.
For scalar conservation laws and symmetric systems, the DG method satisfies a cell entropy inequality \cite{Jiang1994}.
However, for general hyperbolic systems, such as the Euler equations, techniques such as flux differencing are required to ensure entropy stability \cite{Fisher2013,Chen2017,Chan2018,Pazner2018b}.
Furthermore, the use of high degree polynomials can introduce oscillations, and therefore limiters or artificial viscosity techniques are often used for bounds preservation, monotonicity, and shock capturing \cite{Qiu2005,Krivodonova2007,Krivodonova2004,Persson2006}.

An alternative to the above stabilization and limiting strategies is an approach developed by Guermond, Popov, and colleagues, based on invariant domain preserving (IDP) discretizations and convex limiting \cite{Guermond2016,Guermond2018,Guermond2019,Kuzmin2020}.
A desirable property for numerical discretizations of hyperbolic conservation laws is \textit{invariant domain preservation}: if the exact solution to the conservation law lies in a convex invariant set, then the numerical solution should as well~\cite{Guermond2016}.
This is a generalization of the concept of a discrete maximum principle, and will ensure that the the discretization is bounds preserving, positivity preserving, and non-oscillatory.
Suitable low-order invariant domain preserving (IDP) discretizations have been paired with high-order discretizations using convex limiting or algebraic flux correction strategies to obtain second-order accurate methods that preserve specific invariant domains \cite{Guermond2018,Guermond2019,Kuzmin2020}.

In this work, we develop high-order discontinuous Galerkin methods that satisfy invariant domain preserving properties using a convex limiting strategy.
The limiting strategy makes use of a novel sparse low-order IDP method whose stencil does not grow with the polynomial degree of the corresponding high-order method.
Crucially, the accuracy of the low-order method does not degrade as the polynomial degree of the high-order method is increased, as is observed to occur with more naive graph visocity approaches.
Related strategies for sparsifying the convective operator for Bernstein basis finite element methods were previously developed by Kuzmin and colleagues \cite{Kuzmin2020a}.
The flux-corrected method is obtained by performing an efficient, dimension-by-dimension subcell flux correction procedure, blending the low-order IDP method with the high-order target DG method.
The resulting method is conservative, and can satisfy any number of constraints on quasiconcave functionals specified by the user (cf.~\cite{Guermond2018}).
Since the accuracy of the low-order method does not decrease with the polynomial degree of the target method, we observe more accurate results using higher-order methods with a fixed number of degrees of freedom, even on problems with discontinuous solutions.
This method can also be combined with a subcell resolution smoothness indicator to alleviate peak clipping effects near smooth extrema.

The structure of this paper is as follows.
In Section \ref{sec:formulation}, we formulate the high-order DG discretization, and state some key properties.
In Section \ref{sec:idp}, we introduce a new low-order sparsified discretization that can be rendered invariant domain preserving using a graph viscosity approach.
We develop a subcell flux correction strategy for blending the high-order (target) method and the low-order IDP method in Section \ref{sec:fct}.
As specific examples, applications to the linear advection equation with variable velocity field and the Euler equations of gas dynamics are discussed.
A number of numerical test cases demonstrating the effectiveness of the method on both scalar equations and systems of hyperbolic conservation laws are presented in Section \ref{sec:results}.
Finally, we end with some concluding remarks in Section \ref{sec:conclusions}.

\section{Governing equations and discretization}
\label{sec:formulation}

Consider a system of hyperbolic conservation laws,
\begin{equation} \label{eq:cons-law}
   \frac{\partial\bm u}{\partial t} + \nabla \cdot \bm F(\bm u) = 0,
\end{equation}
with solution $\bm u(\bm x, t) \in \R^{\nc}$, $\bm x \in \Omega \subseteq \R^d$.
The flux function is given by $\bm F(\bm u(\bm x)) \in \R^{d \times \nc}$.
The spatial dimension is denoted $d$, and the number of solution components is $\nc$.
The initial conditions are given by $\bm u(\bm x, 0) = \bm u_0(\bm x)$.
Closely associated with the problem \eqref{eq:cons-law} is the following one dimensional Riemann problem, which will be important both for the definition of invariant domain preservation, and for the formulation of the discontinuous Galerkin discretization.
Let $\bm u^-, \bm u^+ \in \R^{\nc}$ be a pair of admissible states, and let $\bm n \in \R^d$ by any unit vector.
We assume that the Riemann problem
\begin{equation} \label{eq:riemann-problem}
    \partial_t \bm u + \partial_x \left( \bm F(\bm u) \cdot \bm n \right) = 0,
    \qquad
    \bm u(x, 0) = \begin{cases}
        \bm u^- & \text{if $x < 0$,} \\
        \bm u^+ & \text{if $x > 0$,}
    \end{cases}
\end{equation}
has a unique self-similar entropy solution $\bm u(x,t) = \bm u_{\bm n, \bm u^-, \bm u^+}(x,t)$.
We denote by $\lambda_{\rm max}$ the maximum wave speed for \eqref{eq:riemann-problem}, for which we have $\bm u(x, t) = \bm u^-$ if $x/t \leq -\lambda_{\rm max}$, and $\bm u(x, t) = \bm u^+$ if $x/t \geq \lambda_{\rm max}$.

We recall from \cite{Guermond2016,Guermond2018} the following definition of an invariant set for \eqref{eq:cons-law}.
\begin{defn}
  We say a set $\mathcal{A} \subseteq \R^{\nc}$ is an \textbf{invariant set} for \eqref{eq:cons-law} if, for any pair of states $\bm u^-, \bm u^+ \in \mathcal{A}$, and for any unit vector $\bm n \in \R^d$, the average of the solution $\bm u(x,t) = \bm u_{\bm n, \bm u^-, \bm u^+}(x, t)$ to the Riemann problem \eqref{eq:riemann-problem}, given by
  \[
    \frac{1}{2 t \lambda_{\rm max}} \int_{-\lambda_{\rm max}t}^{\lambda_{\rm max}t} \bm u(x,t) \, dx
  \]
  remains in $\mathcal{A}$ for any $t > 0$.
\end{defn}
For example, the maximum principle implies that any interval $[a,b] \subseteq \R$ is an invariant set for scalar conservation laws.
For the Euler equations, the set of states with positive density, positive internal energy, and satisfying a minimum principle on specific entropy is a convex invariant set.
Additional examples of invariant sets for systems of conversation laws are given in \cite{Guermond2016,Guermond2019}.
It will be desirable to construct discretizations of \eqref{eq:cons-law} that are \textit{invariant domain preserving} (IDP), meaning that if the approximate solution lies in a convex invariant set $\mathcal{A}$ at some time $t_0$, then the solution will remain in $\mathcal{A}$ for all time $t > t_0$.

The strategy we present here for developing IDP discretizations for \eqref{eq:cons-law} is as follows.
We first formulate a high-order discontinuous Galerkin (DG) discretization that will serve as a \textit{target} scheme.
This discretization will in general not be invariant domain preserving.
We will then modify this high-order discretization to generate a robust low-order discretization.
These modifications take the form of first sparsifying the method to reduce the size of the stencil, and then adding a \textit{graph viscosity} term (cf.~\cite{Guermond2016}), which guarantees that the resulting discretization is invariant domain preserving.
Finally, a subcell flux corrected transport (FCT) technique is used to blend the low-order IDP method and the high-order target method in such a way that specified convex invariant sets are preserved.

\subsection{DG formulation}

We begin by defining the high-order DG discretization for equation \eqref{eq:cons-law}.
The spatial domain $\Omega$ is discretized with a mesh of tensor-product elements denoted $\T_h$.
Each element $K \in \T_h$ is the image of the reference element $\Ref = [0,1]^d$ (the unit cube in $d$ dimensions) under a transformation mapping $T_K$.
To define the standard discontinuous Galerkin finite element space $V_h$, first consider the space $\Q_p([0,1]^d)$ defined on the reference element $\Ref$ consisting of all multivariate polynomials of degree at most $p$ in each variable.
On a given element $K\in\T_h$, we define the space $\Q_p(K)$ to be spanned by functions $\phi \circ T_K^{-1}$, where $T_K : \Ref \to K$ is the element mapping, for all $\phi \in \Q_p([0,1]^d)$.
Then, the space $V_h$ is defined as
\begin{equation}
   V_h = \left\{
      v_h \in L^1(\Omega) : v_h|_K \in \Q_p(K) \text{ for all } K \in \T_h
   \right\}.
\end{equation}
Note that no continuity is enforced between adjacent elements.
To represent approximate solutions to \eqref{eq:cons-law}, we also consider the vector version of this space $\bm V_h = [V_h]^{\nc}$.

We use a nodal Gauss-Lobatto basis for the space $V_h$.
Let $\xi_i$ denote the Gauss-Lobatto points in the interval $[0,1]$, and let $\phi_i$ denote the Lagrange interpolating polynomial satisfying $\phi_i(\xi_j) = \delta_{ij}$, where $\delta_{ij}$ is the Kronecker delta.
These functions form a basis for the space $\Q_p([0,1])$ in one dimension.
The basis for $\Q_p([0,1]^d)$ is formed by taking the tensor product of the one-dimensional basis.
To be precise, we define a function $\Phi_{\bm i}(\bm x) = \prod_{j=1}^d \phi_{i_j}(x_i)$, where $\bm i$ denotes the multi-index $\bm i = (i_1, \ldots, i_d)$, and $\bm x = (x_1, \ldots, x_d)$.
This basis can also be seen to be the nodal interpolation basis corresponding to the Cartesian product of the one-dimensional Gauss-Lobatto nodes.

We approximate the solution to \eqref{eq:cons-law} by $\bm u_h \in \bm V_h$, multiply the equation by a test function $\bm v_h \in \bm V_h$, and integrate over the domain $\Omega$, integrating the flux term by parts over each element $K \in \T_h$.
Because the space $V_h$ is discontinuous, the fluxes $\bm F(\bm u_h)$ are not well-defined on element interfaces.
Consider two neighboring elements, $K^-$ and $K^+$.
Let $\bm u_h^-$ denote the trace of $\bm u_h$ from within $K^-$, and similarly for $\bm u_h^+$.
We therefore introduce a single-valued numerical flux function $\Fhat(\bm u_h^-, \bm u_h^+, \bm n^-)$, obtaining the weak formulation
\begin{equation} \label{eq:weak} \tag{WF}
   \int_\Omega \partial_t \bm u_h \cdot \bm v_h \,d\bm x
   - \sum_{K\in\T_h} \int_K \bm F(\bm u_h) : \nabla \bm v_h \, d\bm x
   + \sum_{K^- \in \T_h} \int_{\partial K^-} \Fhat(\bm u_h^-, \bm u_h^+, \bm n^-) \cdot \bm v_h^- \,ds = 0.
\end{equation}
Integrating the second term on the left-hand side once more by parts, element-by-element one obtains what is know as the strong formulation,
\begin{equation}
\label{eq:strong} \tag{SF}
   \int_\Omega \partial_t \bm u_h \cdot \bm v_h \,d\bm x
   + \sum_{K\in\T_h}\int_K \left( \nabla \cdot \bm F(\bm u_h) \right) \cdot \bm v_h \, d\bm x
   + \sum_{K^- \in \T_h} \int_{\partial K^-} \left( \Fhat(\bm u_h^-, \bm u_h^+, \bm n^-) - \bm F(\bm u_h^-) \cdot \bm n^- \right) \cdot \bm v_h^- \,ds = 0,
\end{equation}
Note that at the continuous level, the formulations \eqref{eq:weak} and \eqref{eq:strong} are equivalent.
However, after discretization, they may differ because of inexact integration.

For the purposes of discretization, it is convenient to transform the integrals in both \eqref{eq:weak} and \eqref{eq:strong} to integrals over the reference element $\Ref$.
This is done using a standard transformation of the governing equation \eqref{eq:cons-law} from each element $K$ to the reference element $\Ref$.
Consider a given element $K\in\T_h$ with transformating mapping $T_K$.
Let $J$ denote the Jacobian matrix of the mapping $T_K$.
The inverse of the Jacobian is used to define the contravariant fluxes
\begin{equation} \label{eq:contravariant}
     \tilde{\bm F}_i = \det(J) \sum_{j=1}^d J^{-1}_{ij} \bm F_j.
\end{equation}
Then, on the reference element, the solution $\bm u_h$ evolves according to the transformed conservation law
\begin{equation} \label{eq:transformed-law}
     \frac{\partial \bm{\tilde{u}}(\bm \xi)}{\partial t}
     + \nabla \cdot \bm{\tilde{F}} (\bm{\tilde{u}}(\bm \xi)) = 0,
     \qquad \bm\xi \in \Ref,
\end{equation}
where $\bm{\tilde{u}} = \det(J) \bm u$.

\subsubsection{Numerical flux functions}
\label{sec:numerical-flux}

An important aspect of DG methods is the choice of numerical flux function $\Fhat$.
The numerical flux functions are typically chosen to be either exact or approximate Riemann solvers for the one-dimensional Riemann problem in the normal direction at element interfaces.
In this work, we will make use of the simple local Lax-Friedrichs numerical flux function.
The reason for this choice is that the Lax-Friedrichs flux is compatible with the graph viscosity used to ensure that the low order discretization is invariant domain preserving, as will be discussed in greater detail in Section \ref{sec:idp}.
The Lax-Friedrichs flux is defined by
\begin{equation} \label{eq:lax-friedrichs}
    \Fhat(\bm u_h^-, \bm u_h^+, \bm n^-)
    = \frac{1}{2} \left( \bm F(\bm u_h^-) + \bm F(\bm u_h^+) \right) \cdot \bm n^-
    - \frac{\lambda}{2} \left( \bm u_h^+ - \bm u_h^-  \right)
\end{equation}
where $\lambda \geq \lambda_{\max}(\bm u_h^-, \bm u_h^+, \bm n^-)$ is an upper bound for the maximum wave speed of the Riemann problem
\begin{equation}
    \partial_t \bm u + \partial_x \left( \bm F(\bm u) \cdot \bm n^- \right) = 0,
    \qquad
    \bm u(x, 0) = \begin{cases}
        \bm u_h^- & \text{if $x < 0$,} \\
        \bm u_h^+ & \text{if $x > 0$.}
    \end{cases}
\end{equation}

\subsection{Collocated and one-dimensional operators}

The discontinuous Galerkin spectral element method (DG-SEM) is distinguished from other DG methods by a specific choice of quadrature rule and basis.
We proceed by choosing a nodal basis for the space $V_h$ and then approximating the integrals in \eqref{eq:weak} and \eqref{eq:strong} with collocated quadrature rules.
Typically, Gauss-Legendre or Gauss-Lobatto nodes are chosen for the basis functions.
In this work, we will solely make use of Gauss-Lobatto nodes and quadrature.

Due to the use of tensor-product basis and quadrature, the DG operators possess a Kronecker-product structure~\cite{Pazner2018e}.
Since the nodal points and quadrature points are collocated, there is no need for an interpolation operator.
The one-dimensional mass matrix on the reference element, denoted $\M_\oned$, is given by a diagonal matrix with quadrature weights on the diagonal.
The weighted one-dimensional differentiation matrix $\D_\oned$ is obtained by evaluating the derivatives of the basis functions at the nodal points and multiplying by quadrature weights,
\begin{equation}
   (\D_\oned)_{ij} = w_i \phi_j'(\xi_i).
\end{equation}
We will often make use of the following two simple properties of the differentiation matrix:

\begin{prop}
The weighted differentiation matrix satisfies the following two useful properties, which we will make use of extensively in this work:
\begin{align}
   \tag{P1} \label{eq:zero-sum} \sum_j (\D_\oned)_{ij} &= 0 \qquad\text{for all $i$}, \\
   \tag{P2} \label{eq:sbp} \sum_i (\D_\oned)_{ij} &= \begin{cases} -1, & j = 1 \\ 1, & j = p+1 \\ 0 & \text{otherwise} \end{cases}
\end{align}
\begin{itemize}
   \item The property \eqref{eq:zero-sum} ensures that the resulting method is conservative, and follows from $\sum_j \phi_j \equiv 1$.
   \item The property \eqref{eq:sbp} is know as the summation-by-parts (SBP) property, and is a consequence of the accuracy of the Gauss-Lobatto quadrature for polynomials of degree $2p-1$.
\end{itemize}
\end{prop}

On the reference element, the local mass and differentiation operators can be obtained through Kronecker products.
For instance, for $d=2$ we have
\begin{equation} \label{eq:kronecker}
   \M_\Ref = \M_\oned \otimes \M_\oned, \qquad
   \D_{1,\Ref} = \M_\oned \otimes \D_\oned, \qquad
   \D_{2, \Ref} = \D_\oned \otimes \M_\oned.
\end{equation}
We can also define the left endpoint evaluation matrix $\B_{0,\oned}$, which is zero except for the first entry of the diagonal, which takes value one, and likewise the right endpoint evaluation matrix $\B_{1,\oned}$, which is zero except for the last entry of the diagonal.
Given these definitions, integrals over the boundary of the reference element $\partial\Ref$ can be computed using Kronecker products,
\begin{equation}
   \begin{aligned}
      \B_{x=0,\Ref} = \M_\oned \otimes \B_{0,\oned}, &&
      \B_{x=1,\Ref} = \M_\oned \otimes \B_{1,\oned}, \\
      \B_{y=0,\Ref} = \B_{0,\oned} \otimes \M_\oned, &&
      \B_{y=1,\Ref} = \B_{1,\oned} \otimes \M_\oned.
   \end{aligned}
\end{equation}

\subsection{Metric terms and transformed operators}

Having defined the mass, differentiation, and boundary operators on the reference element as above, we now wish to transform these operators to act in physical coordinates.
Consider a fixed element $K\in\T_h$ with transformation mapping $T_K$ and Jacobian $J$.
Let $\J$ denote the values of the Jacobian matrix of $T_K$ evaluated at the Gauss-Lobatto nodal points.
These terms are evaluated using the freestream-preserving procedure descried in \cite{Kopriva2006}.
Let $\G$ denote the diagonal matrix whose entries are given by $\det(\J)$.
Similarly, let $\Jinv_{ij}$ denote the ($d\times d$ block) diagonal matrix whose entries are given by $\det(\J) \J^{-1}_{ij}$.
Notice that the elemental mass matrix corresponding to the element $K$ is given by $\M_K = \M_\Ref \G$.

Let $\uu$ denote the vector of coefficients (i.e.\ nodal values) of $\bm u_h$ on $K$ and let $\FF$ denote the vector of values of $\bm F(\bm u_h)$ evaluated at the nodal points (i.e.\ $\FF = \bm F(\uu)$).
For each face $e \in \partial K$, let $\FF_e\cdot\n$ denote the values of $\bm F(\bm u_h)$ evaluated at the $(d-1)$-dimensional Gauss-Lobatto nodes on face $e$ (where the trace of $\bm u_h$ is taken from within $K$), dotted with the scaled normal vector $\n$ facing outwards from $e$. Likewise $\FFhat_e$ denotes the nodal values of $\Fhat(\bm u_h^-, \bm u_h^+, \bm n^-)$.

The contravariant fluxes defined by \eqref{eq:contravariant} are given by
\begin{equation}
     \tilde{\FF}_i = \sum_{j=1}^d \Jinv_{ij} \FF_j.
\end{equation}
Therefore, the term $\int_K \tilde{\bm F} : \nabla_h \bm v_h \, d\bm x$ in the weak formulation is discretized as
\begin{equation}
     \sum_{i=1}^d \D_{i,\Ref}^\tr \tilde{\FF}_i = \sum_{i=1}^d \D_{i,\Ref}^\tr \sum_{j=1}^d \Jinv_{ij} \FF_j,
\end{equation}
and the divergence term $\int_K (\nabla \cdot \tilde{\bm F} ) \cdot \bm v_h \, d\bm x$ is discretized as
\begin{equation}
     \sum_{i=1}^d \D_{i,\Ref} \tilde{\FF}_i = \sum_{i=1}^d \D_{i,\Ref} \sum_{j=1}^d \Jinv_{ij} \FF_j.
\end{equation}
The weak form \eqref{eq:weak} on element $K$ can therefore be written
\begin{equation} \label{eq:dg-weak}
     \M_K \uu_t
     - \sum_{i=1}^d \D_{i,\Ref}^\tr \sum_{j=1}^d \Jinv_{ij} \FF_j
     + \sum_{e\in\partial K} \B_{e,\Ref} \FFhat_e = 0.
\end{equation}
Similarly, the strong form \eqref{eq:strong} is given by
\begin{equation} \label{eq:dg-strong}
     \M_K \uu_t
     + \sum_{i=1}^d \D_{i,\Ref} \sum_{j=1}^d \Jinv_{ij} \FF_j
     + \sum_{e\in\partial K} \B_{e,\Ref} (\FFhat_e - \FF_e\cdot\n) = 0.
\end{equation}

\subsection{Conservation and constant preservation}
\label{sec:conservation}

The governing equation \eqref{eq:cons-law} satisfies the following two simple properties:
\begin{itemize}
     \item \textbf{(Conservation).}
     Assuming periodic or compactly supported boundary conditions, $\int_\Omega \frac{\partial \bm u}{\partial t} \, d\bm x = 0$.
     \item \textbf{(Constant preservation).}
     If $\bm F(\bm u)$ is spatially constant, then $\frac{\partial \bm u}{\partial t} = 0$.
     In particular, if $\bm F$ depends only on $\bm u$ (and not, for example, on the spatial variable $\bm x$), then if $\bm u$ being spatially constant implies that $\frac{\partial \bm u}{\partial t} = 0$.
\end{itemize}
We would like the discretization to satisfy the analogous properties at the discrete level.

\subsubsection{Conservation}
First, we consider conservation.
The analogous statement at the discrete level is that
\begin{equation}
     \sum_{K\in\T_h} \mathtt{1}^\tr \M_K \uu_t = 0,
\end{equation}
where $\mathtt{1}$ is a vector of all ones.
First, note that since the numerical flux function $\Fhat$ is single-valued, we have
\begin{equation}
     \sum_{K\in\T_h}\sum_{e\in\partial K} \mathtt{1}^\tr \B_{e,\Ref} \FFhat_e = 0.
\end{equation}
The remaining terms in both the strong form and weak form are local to each element and do not involve contributions from face neighbors.
In the weak form \eqref{eq:dg-weak}, we see that $\mathtt{1}^\tr D_{i,\Ref}^\tr = 0$ by property \eqref{eq:zero-sum}.
We can then conclude that
\begin{equation}
     \sum_{K\in\T_h}\mathtt{1}^\tr \M_K \uu_t
     = \sum_{K\in\T_h} \left( \sum_{i=1}^d \mathtt{1}^\tr \D_{i,\Ref}^\tr \sum_{j=1}^d \Jinv_{ij} \FF_j
     - \sum_{e\in\partial K} \mathtt{1}^\tr \B_{e,\Ref} \FFhat_e \right) = 0,
\end{equation}
proving the conservation property for the discretized weak form.

For the discretized strong form, we make use of property \eqref{eq:sbp}.
For any $K\in\T_h$
\begin{equation} \label{eq:sbp-strong}
     \mathtt{1}^\tr \sum_{i=1}^d \D_{i,\Ref} \tilde{\FF}_i
     = \mathtt{1}^\tr \sum_{e \in \partial K} \B_{e,\Ref} \FF_e\cdot\n,
\end{equation}
and we therefore obtain conservation for the strong form as well:
\begin{equation}
     \sum_{K\in\T_h}\mathtt{1}^\tr \M_K \uu_t
     = - \sum_{K\in\T_h} \left( \sum_{i=1}^d \mathtt{1}^\tr \D_{i,\Ref} \sum_{j=1}^d \Jinv_{ij} \FF_j
     - \sum_{e\in\partial K} \mathtt{1}^\tr \B_{e,\Ref} \left(\FFhat_e - \FF_e\cdot\n \right) \right) = 0.
\end{equation}

We summarize the above arguments in the following proposition:
\begin{prop}
  Let $\uu_t$ satisfy either the discretized weak form \eqref{eq:dg-weak} or strong form \eqref{eq:dg-strong}.
  Then, the following conservation property holds:
  \begin{equation}
    \sum_{K\in\T_h} \mathtt{1}^\tr \M_K \uu_t = 0.
  \end{equation}
\end{prop}

\subsubsection{Constant preservation}

Now we turn to the property of constant preservation.
We supposed that the flux $\bf F$ is spatially constant, and therefore $\Fhat$ is also everywhere equal to the same constant.
As a result, this property is easily proven for the strong form of the discretization: $\FFhat_e - \FF_e\cdot\n = 0$, and so the third term on the left-hand side of \eqref{eq:dg-strong} is zero.
Therefore, the discretization will preserve constants if the following identity holds:
\begin{equation} \label{eq:metric}
     \sum_{i=1}^d \D_{i,\Ref} \Jinv_{ij} = 0 \quad \text{for all $j$.}
\end{equation}
This so-called \emph{metric identity} is discussed in detail in \cite{Kopriva2006}, and can be enforced through proper evaluation of the entries of $\Jinv$.
In this case, we have the following result:
\begin{prop}
  Let $\uu_t$ satisfy the discretized strong form \eqref{eq:dg-strong}.
  Furthermore, suppose that $\bm F(\bm u_h)$ is spatially constant.
  Then, $\uu_t = 0$.
\end{prop}

On the other hand, the discretized weak form given by \eqref{eq:dg-weak} will not in general satisfy the constant preservation property (e.g.\ on curved meshes).
For this reason, in the remainder of this paper, we make use of the strong form discretization \eqref{eq:dg-strong} to obtain the unlimited {\it target scheme}.

\section{Construction of the IDP low-order method}
\label{sec:idp}

We now modify the discretization \eqref{eq:dg-strong} with the goal of obtaining a method which is \emph{invariant domain preserving} (IDP).
In Guermond and Popov \cite{Guermond2016} this is done by adding a \emph{graph viscosity} term that is based on the \emph{guaranteed maximum speed} (GMS) of the hyperbolic system.
One potential issue with this approach is that the amount of graph viscosity added to the discretization increases as the size of the discrete stencil increases.
As a result, applying this technique to high-order methods with large stencils results in very dissipative methods and typically poor-quality results, as observed in \cite{Lohmann2017}, and further illustrated in Section \ref{sec:unsparsified}.
In order to address this issue, we are interested in creating an IDP discretization that is compatible in a certain sense with \eqref{eq:dg-strong}, yet based on a more compact stencil.
Because the graph viscosity term is itself first-order accurate, the underlying sparse discretization is not required to be high-order accurate.

We make the following simple modification to the DG-SEM method described above.
Notice that the wide stencil of the high-order method is a consequence of the fact that the one-dimensional differentiation matrix $\D_\oned$ is dense.
We therefore replace $\D_\oned$ with a sparser version $\Didp_\oned$ that is first-order accurate.
$\Didp_\oned^\tr$ is obtained by integrating the derivatives of piecewise linear basis functions on the mesh defined by the Gauss-Lobatto points in the interval $[0,1]$.
$\Didp_\oned$ is therefore given by
\begin{equation} \label{eq:didp}
     \Didp_\oned
     = \left(\begin{array}{ccccccc}
          -\frac{1}{2} & \frac{1}{2} & 0 & 0 & 0 & \cdots & 0 \\
          -\frac{1}{2} & 0 & \frac{1}{2} & 0 & 0 & \cdots & 0 \\
          0 & -\frac{1}{2} & 0 & \frac{1}{2} & 0 & \cdots & 0 \\
          \vdots & \vdots & \vdots & \vdots & \vdots & \ddots & \vdots
     \end{array}\right).
\end{equation}
Then, for each $i=1,\ldots,d$, we construct operators $\Didp_{i,\Ref}$ by replacing $\D_\oned$ with $\Didp_\oned$ in the definition \eqref{eq:kronecker}.
It is easy to see from \eqref{eq:didp} that $\Didp$ also satisfies the row and column-sum properties \eqref{eq:zero-sum} and \eqref{eq:sbp}.
The modified operators $\Didp_{i,\Ref}$ then replace the standard DG-SEM derivative operators $\D_{i,\Ref}$ in the formulation \eqref{eq:dg-weak}, in order to obtain the following modified discretization:
\begin{equation} \label{eq:dg-lo}
     \M_K \uu_t
     - \sum_{i=1}^d \Didp_{i,\Ref}^\tr \sum_{j=1}^d \Jinv_{ij} \FF_j
     + \sum_{e\in\partial K} \B_{e,\Ref} \FFhat_e = 0.
\end{equation}

\subsection{Conservative correction} \label{sec:correction}

For our purposes, it is important to ensure that the modified formulation \eqref{eq:dg-lo} satisfies the conservation and constant preservation properties described in section \ref{sec:conservation}.
Because \eqref{eq:dg-weak} is based on the weak formulation, the conservation property follows immediately from the zero-sum property \eqref{eq:zero-sum} (and the fact that the numerical fluxes $\Fhat$ are unique).
However, the constant preservation property will not hold in general.
The reason for this is that the modified differentiation matrix cannot differentiate exactly the metric terms in $\Jinv$, and therefore the modified discrete metric identities no longer hold.

To remedy this issue, we make use of modified metric terms $\Jinvidp$ which are $\O(h)$ perturbations of the high-order metric terms $\Jinv$, but are designed such that the first-order discretization satisfies the metric identities.
Since the modified low-order method is itself only first-order accurate, using an $\O(h)$ perturbation of the metric terms is acceptable.

The modified metric terms $\Jinvidp$ are constructed as follows.
We set $\Jinvidp = \Jinv + \mathtt{C}$, where $\mathtt{C}$ are block diagonal correction matrices.
For each element, we would like to enforce the identity
\begin{equation} \label{eq:idp-sbp}
     \sum_{i=1}^d \Didp_{i,\Ref}^\tr \Jinvidp_{ij} \mathtt{1}
     = \sum_{e\in\partial K} \B_{e,\Ref} \n_{e,j} \mathtt{1} \qquad \text{for $1 \leq j \leq d$},
\end{equation}
where $\n_{e,j}$ denotes the $j$th component of the outward-facing normal at the Gauss-Lobatto points of face $e$.
This is equivalent to the following underdetermined system of equations:
\begin{equation}
     \left(\begin{array}{ccc}
          \Didp_{1,\Ref}^\tr & \cdots \Didp_{d,\Ref}^\tr
     \end{array}\right)
     \left(\begin{array}{c}
          \Jinvidp_{1j} \mathtt{1} \\
          \vdots \\
          \Jinvidp_{dj} \mathtt{1}
     \end{array}\right)
     = \sum_{e \in \partial K} \B_{e, \Ref} \n_{e,j} \mathtt{1}.
\end{equation}
Solving for the correction matrices $\mathtt{C}$, we obtain the system
\begin{equation} \label{eq:metric-perturbation}
    \left(\begin{array}{ccc}
        \Didp_{1,\Ref}^\tr & \cdots \Didp_{d,\Ref}^\tr
    \end{array}\right)
    \left(\begin{array}{c}
        \mathtt{C}_{1j} \mathtt{1} \\
        \vdots \\
        \mathtt{C}_{dj} \mathtt{1}
    \end{array}\right)
    = \sum_{e \in \partial K} \B_{e, \Ref} \n_{e,j} \mathtt{1}
    - \left(\begin{array}{ccc}
        \Didp_{1,\Ref}^\tr & \cdots \Didp_{d,\Ref}^\tr
    \end{array}\right)
    \left(\begin{array}{c}
        \Jinv_{1j} \mathtt{1} \\
        \vdots \\
        \Jinv_{dj} \mathtt{1}
    \end{array}\right).
\end{equation}
Note that the differentiation matrix on the left-hand side has a null-space consisting of all constant functions.
In order for a solution to exist, the right-hand side must be orthogonal to this null-space.
In other words, the entries of the right-hand side vector must sum to zero.
This can be seen to be true for the second term on the right-hand side by a simple application of the zero-sum property \eqref{eq:zero-sum}.
In order to show this property for the first term on the right-hand side, we make use of the summation-by-parts property \eqref{eq:sbp-strong}:
\begin{equation}
    \mathtt{1}^\tr \sum_{e \in \partial K} \B_{e,\Ref} \n_{e,j}
    = \mathtt{1}^\tr \sum_{i=1}^d \D_{i,\Ref} \Jinv_{ij}.
\end{equation}
Then, the metric identity \eqref{eq:metric} implies that this term is equal to zero.
Therefore a solution to \eqref{eq:metric-perturbation} exists.
On each element, the minimum norm solution to \eqref{eq:metric-perturbation} is computed for the perturbed metric terms.
This can be performed using e.g.\ the singular value decomposition, which is performed once as an offline precomputation on the reference element, and then simply reused for each element $K\in\T_h$.
The right-hand side (as well as the metric terms themselves) scale as $\O(h)$, and therefore the entrywise error satisfies $| \Jinvidp - \Jinv | = \O(h)$.

Using the above procedure to construct the modified metric terms, we obtain the low-order discretization
\begin{equation} \label{eq:dg-lo-modified}
     \M_K \uu_t
     - \sum_{i=1}^d \Didp_{i,\Ref}^\tr \sum_{j=1}^d \Jinvidp_{ij} \FF_j
     + \sum_{e\in\partial K} \B_{e,\Ref} \FFhat_e = 0.
\end{equation}
As a consequence of the choice of metric terms, this discretization satisfies the following properties.
\begin{prop}
  Let $\uu_t$ satisfy \eqref{eq:dg-lo-modified}. Then, the following two properties hold:
  \begin{itemize}
    \item Conservation:
    $\sum_{K\in\T_h} \mathtt{1}^\tr \M_K \uu_t = 0.$
    \item Constant preserving: if $\bm F(\bm u_h)$ is spatially constant then $\uu_t = 0$.
  \end{itemize}
\end{prop}

Next, the discretization \eqref{eq:dg-lo} is modified to render the resulting method invariant domain preserving (IDP).
Because of the local nature of the DG method, and because of the choice of Lax-Friedrichs numerical flux (cf.\ Section \ref{sec:numerical-flux}), it is possible to perform this modification in an entirely local fashion.
First, we rewrite \eqref{eq:dg-lo} in a form that will be more convenient for our purposes.
Note that by definition of the Lax-Friedrichs flux (equation \eqref{eq:lax-friedrichs}), the term of the form $\sum_{e\in\partial K} \B_{e,\Ref} \FFhat_e$ can be written as
\begin{equation} \label{eq:lf-alternative}
    \frac{n_f(i)}{2} \wn_{ii} \cdot \FF_i +
    \sum_{j\in\mathcal{B}(i) \setminus \{i\}} \left( \frac{1}{2} \wn_{ij} \cdot \FF_j - \| \wn_{ij} \|_{\ell^2} \frac{\lambda}{2} (\uu_j - \uu_i)  \right),
\end{equation}
where $n_f(i)$ is the number of faces on which the $i$th node lies (i.e.\ depending on if $i$ is an interior node or if it lies on a face, edge, or corner of the element), $\wn_{ii}$ and $\wn_{ij}$ are weighted normal vectors, and $\mathcal{B}(i)$ is the index set consisting of all nodes $j$ that are face neighbors of $i$.
The vector $\wn_{ij}$ is equal to the normal vector evaluated at node $j$, pointing outwards from the element to which node $i$ belongs, weighted by the element of surface area and face quadrature weight.
If the node $i$ does not lie on an element face, then $n_f(i) = 0$ and $\mathcal{B}(i) = \varnothing$, and in this case we leave $\wn_{ij}$ undefined.
Inserting this expression into \eqref{eq:dg-lo} and writing the volume terms in terms of coefficients $\mathtt{c}_{ij}$, we obtain the following equivalent formulation:
\begin{equation} \label{eq:dg-lo-cij}
    \mathtt{m}_i \partial_t \uu_i
    + \frac{n_f(i)}{2} \wn_{ii} \cdot \FF_i
    - \sum_{j \in \mathcal{E}(i)} \bm{\mathtt{c}}_{ij} \cdot \FF_j
    + \sum_{j\in\mathcal{B}(i) \setminus \{i\}} \left( \frac{1}{2} \wn_{ij} \cdot \FF_j - \| \wn_{ij} \|_{\ell^2} \frac{\lambda}{2} (\uu_j - \uu_i)  \right) = 0,
\end{equation}
where $\mathtt{m}_i$ is the $i$th diagonal entry of the mass matrix, and $\mathcal{E}(i)$ is the set of all indices in the stencil of $i$ within the same element.
For any $j \in \mathcal{E}(i)$, we define the symmetric \emph{graph viscosity coefficients} $\mathtt{d}_{ij}$ by
\begin{equation} \label{eq:dij}
    \mathtt{d}_{ij} = \max\left\{
    \lambda_{\max}(\uu_i, \uu_j, \mathtt{n}_{ij}) \| \mathtt{c}_{ij} \|_{\ell^2},
    \lambda_{\max}(\uu_j, \uu_i, \mathtt{n}_{ji}) \| \mathtt{c}_{ji} \|_{\ell^2}
    \right\},
\end{equation}
where $\mathtt{n}_{ij} = \mathtt{c}_{ij}/\| \mathtt{c}_{ij} \|_{\ell^2}$.
For convenience of notation, we will use the convention that $\mathtt{d}_{ii} = -\sum_{i\neq j\in\mathcal{E}(i)} \mathtt{d}_{ij}$.
Having defined the quantities $\mathtt{d}_{ij}$, the discretization \eqref{eq:dg-lo-cij} is modified with the addition of a viscosity term
\begin{equation} \label{eq:dg-graph-vis-long}
  \mathtt{m}_i \partial_t \uu_i
  + \frac{n_f(i)}{2} \wn_{ii} \cdot \FF_i
  - \sum_{j \in \mathcal{E}(i)} \bm{\mathtt{c}}_{ij} \cdot \FF_j
  + \sum_{j\in\mathcal{B}(i) \setminus \{i\}} \left( \frac{1}{2} \wn_{ij} \cdot \FF_j - \| \wn_{ij} \|_{\ell^2} \frac{\lambda}{2} (\uu_j - \uu_i)  \right)
  - \sum_{j \in \mathcal{E}(i)}\mathtt{d}_{ij} ( \uu_j - \uu_i ) = 0.
\end{equation}
The above expression is simplified by combining the volume and boundary terms.
Let $\mathcal{N}(i) = \mathcal{E}(i) \cup \mathcal{B}(i)$, and let $\hat{\mathtt{c}}_{ij} = \mathtt{c}_{ij} + \mathtt{b}_{ij}$, where $\mathtt{b}_{ij}$ denotes the corresponding coefficient of $\FF_j$ from the boundary terms.
Here we use the convention that $\mathtt{c}_{ij} = 0$ if $j \notin \mathcal{E}(i)$ and likewise $\mathtt{b}_{ij} = 0$ if $j \notin \mathcal{B}(i)$.
At this point, notice that the Lax-Friedrichs term may be combined with the graph viscosity by appropriately defining the viscosity coefficients.
Let $\hat{\mathtt{d}}_{ij}$ for $j \neq i$ be given by
\begin{equation}
  \hat{\mathtt{d}}_{ij} = \begin{cases}
    \mathtt{d}_{ij}, & \quad \text{if } j \in \mathcal{E}(i), \\
    \frac{1}{2} \lambda \| \wn_{ij} \|_{\ell^2}, & \quad \text{if } j \in \mathcal{B}(i).
  \end{cases}
\end{equation}
Note that by this definition is equivalent to replacing $\mathtt{c}_{ij}$ with $\hat{\mathtt{c}}_{ij}$ in \eqref{eq:dij}.
Using these definitions, \eqref{eq:dg-graph-vis-long} simplifies to
\begin{equation} \label{eq:dg-idp}
  \mathtt{m}_i \partial_t \uu_i
  - \sum_{j \in \mathcal{N}(i)} \hat{\mathtt{c}}_{ij} \cdot \FF_j
  - \sum_{j \in \mathcal{N}(i)} \hat{\mathtt{d}}_{ij} ( \uu_j - \uu_i ) = 0.
\end{equation}

\begin{prop}
  The discretization given by \eqref{eq:dg-idp} satisfies the following properties.
  \begin{itemize}
    \item Since the low-order discretization is constant preserving,
    $
      \sum_{j \in \mathcal{N}(i)} \hat{\mathtt{c}}_{ij} = 0.
    $
    \item By symmetry of the coefficients $\mathtt{d}_{ij}$, the graph viscosity contributions sum to zero:
    \[
      \sum_{i} \sum_{j \in \mathcal{N}(i)} \hat{\mathtt{d}}_{ij} (\uu_j - \uu_i) = 0.
    \]
  \end{itemize}
\end{prop}

\subsection{Invariant domain preservation}

We now set out to prove that the discretization defined by \eqref{eq:dg-idp} is \emph{invariant domain preserving} (IDP).
We will make use of a strong stability preserving (SSP) Runge-Kutta method for the temporal discretization \cite{Gottlieb2011}.
Such methods can be written as a convex combination of forward Euler steps, and so it suffices to prove the IDP property for a forward Euler update.
We therefore consider the forward Euler discretization of \eqref{eq:dg-idp}, given by
\begin{equation} \label{eq:idp-fe}
  \frac{\mathtt{m}_i}{\Delta t} \uu_i^{n+1}
  = \frac{\mathtt{m}_i}{\Delta t} \uu_i^{n}
  + \sum_{j \in \mathcal{N}(i)} \hat{\mathtt{c}}_{ij} \cdot \FF_j^n
  + \sum_{j \in \mathcal{N}(i)} \hat{\mathtt{d}}_{ij} ( \uu_j^n - \uu_i^n ).
\end{equation}
As in \cite{Guermond2016}, we use the property $\sum_{j \in \mathcal{N}(i)} \hat{\mathtt{c}}_{ij} = 0$ to rewrite \eqref{eq:dg-idp} in the form
\begin{equation} \label{eq:idp-fe-2}
  \frac{\mathtt{m}_i}{\Delta t} \uu_i^{n+1}
  = \uu_i^{n} \left( \frac{\mathtt{m}_i}{\Delta t} - \sum_{\mathclap{i \neq j \in \mathcal{N}(i)}} 2 \hat{\mathtt{d}}_{ij} \right)
  + \sum_{\mathclap{i\neq j\in\mathcal{N}(i)}}
    \big(
      \hat{\mathtt{c}}_{ij} \cdot \left( \FF_j^n - \FF_i^n \right)
      + \hat{\mathtt{d}}_{ij} \left( \uu_j^n + \uu_i^n \right)
    \big).
\end{equation}
At this point we can introduce the so-called ``bar states'' (also referred to as ``intermediate limiting states'', cf.~\cite{Guermond2016}), which are defined by
\begin{equation} \label{eq:bar-states}
    \overline{\uu}_{ij}^{n+1}
        = \frac{1}{2} \left( \uu_i^n + \uu_j^n \right)
        + \frac{\hat{\mathtt{c}}_{ij}}{2 \hat{\mathtt{d}}_{ij}}
            \cdot \left(\FF_j^n - \FF_i^n \right).
\end{equation}
\begin{prop} \label{prop:bar-states}
    Suppose $\mathcal{A}$ is a convex invariant set of \eqref{eq:cons-law} such that $\uu_i^n, \uu_j^n \in \mathcal{A}.$
    Then, the bar states $\overline{\uu}_{ij}$ belong to $\mathcal{A}$.
\end{prop}
\begin{proof}
    See \cite{Guermond2019}.
\end{proof}
Having defined the bar states, we rewrite \eqref{eq:idp-fe-2} as a convex combination
\begin{equation} \label{eq:convex-combination}
    \uu_i^{n+1} = \left( 1 - \sum_{\mathclap{i\neq j\in\mathcal{N}(i)}}
        \frac{2 \Delta t \hat{\mathtt{d}}_{ij}}{\mathtt{m}_i}
    \right) \uu_i^n
    + \sum_{i\neq j\in\mathcal{N}(i)} \left(
         \frac{2 \Delta t \hat{\mathtt{d}}_{ij}}{\mathtt{m}_i}
        \right) \overline{\uu}_{ij}^{n+1}.
\end{equation}
As an immediate consequence of this rewriting is the following proposition.
\begin{prop} \label{prop:idp}
    Suppose the following CFL condition holds:
    \begin{equation}\label{eq:cfl-low}
        \Delta t \leq \min_i \frac{\mathtt{m}_i}{2 \hat{\mathtt{d}}_{ii}},
    \end{equation}
    where $\hat{\mathtt{d}}_{ii} = -\sum_{i\neq j\in\mathcal{N}(i)} \hat{\mathtt{d}}_{ij}$.
    Let $\mathcal{A}$ denote a convex invariant set of \eqref{eq:cons-law}, such that $\uu_i^n \in \mathcal{A}$ for all $i$.
    Then, $\uu_i^{n+1} \in \mathcal{A}$ for all $i$.
\end{prop}

\subsection{Application: linear advection}

Consider the linear, scalar advection equation
\begin{equation}
   u_t(\bm x, t) + \nabla \cdot (\bm\beta(\bm x) u(\bm x, t) ) = 0,
\end{equation}
where $\bm\beta(\bm x) : \Omega \to \R^d$ is a prescribed velocity field.
We assume that the velocity field is divergence free, i.e.\ $\nabla\cdot\bm\beta = 0$.
We are interested in ensuring that the low-order method is \emph{bounds preserving}, i.e.\ if $u(\bm x, 0) \in [a,b]$ for all $\bm x \in \Omega$, then
$u(\bm x, t) \in [a,b]$ for all $t$.

If $\bm\beta$ is spatially constant, then the method described above applies immediately to this case, and the low-order method defined by \eqref{eq:idp-fe} is bounds perserving.
However, if $\bm\beta$ is spatially variable, then some modifications to the above method are required.
The reason for this is that in this case, Proposition \ref{prop:bar-states} may fail to hold.
For instance, if $u_h^n$ is spatially constant (i.e.\ $u^n_i = a$ for all $i$, for some fixed $a$), then we would expect $\overline{\uu}_{ij}^{n+1}$ to be equal to the same constant. However, since $\bm\beta$ is spatially varying, we have, in general, $\FF_j^n \neq \FF_i^n$ and therefore  $\overline{\uu}_{ij}^{n+1} \neq a$.

To avoid this issue, we slightly modify the formulation, using an approach similar to that developed by Kuzmin in \cite{Kuzmin2020}.
We define the modified bar states for the advection equation by
\begin{equation}
    \overline{\uu}_{ij}^{n+1}
        = \frac{1}{2} \left( \uu_i^n + \uu_j^n \right)
        + \frac{\hat{\mathtt{c}}_{ij} \cdot \bm\beta_j}{2 \hat{\mathtt{d}}_{ij}}
            \left(\uu_j^n - \uu_i^n \right),
\end{equation}
where $\bm\beta_i = \bm\beta(\bm x_i)$, and $\bm x_i$ denotes the coordinates of the $i$th node of the mesh.
The modified graph viscosity coefficients for the advection equation are given by
\begin{equation}
    \hat{\mathtt{d}}_{ij} = \max\left\{
    | \hat{\mathtt{c}}_{ij} \cdot \bm\beta_i |,
    | \hat{\mathtt{c}}_{ji} \cdot \bm\beta_j |
    \right\}.
\end{equation}
It is clear from this definition that if $\uu_i^n, \uu_j^n \in [a,b]$ then $\overline{\uu}_{ij} \in [a,b]$.
As a consequence, defining the modified update for $\uu_i^{n+1}$ as the convex combination of bar states given by \eqref{eq:convex-combination}, we see that $\uu_i^{n+1} \in [a,b]$, as desired.
\begin{rem*}
   The property $\sum_{j \in \mathcal{N}(i)} \hat{\mathtt{c}}_{ij} = 0$ was essential to writing the low-order update in terms of the bar states.
   In the context of linear advection, the analogous property is $\sum_{j \in \mathcal{N}(i)} \hat{\mathtt{c}}_{ij}\cdot\bm\beta_j = 0$.
   This is the discrete equivalent to the divergence free constraint, $\nabla\cdot\bm\beta = 0$.
   In order to ensure that this property holds, the conservative correction procedure described in Section \ref{sec:correction} must be modified to take into account the velocity field $\bm\beta$.
\end{rem*}

\subsection{Comparison with unsparsified method}
\label{sec:unsparsified}

As mentioned in Section \ref{sec:idp}, the motivation for introducing the sparsified derivative operators $\Didp$ is that the addition of the graph viscosity term to the ``unsparsified'' operators can result in overly dissipative results when the stencil size is increased.

To illustrate this point, we consider the advection of a sine wave in one dimension.
The domain is taken to be $\Omega = [-1,1]$, the initial condition is $u(x) = \sin(\pi x)$, and periodic boundary conditions are enforced.
We integrate the equation $u_t + u_x = 0$ until a final time of $t=2$, at which point the solution is identical to the initial condition.
We compare the sparsified low-order IDP method to the graph viscosity method applied to the unsparisified DG-SEM operator.
We fix the total number of degrees of freedom to be 128, and consider polynomial degrees 3, 7, 15, and 31 (corresponding to mesh sizes of 32, 16, 8, and 4 elements, respectively).
The results are shown in Figure \ref{fig:low-order-comparison}.
It is immediately clear that increasing the polynomial degree causes a large degradation in quality of the unsparsified methods.
This is because the size of the stencil grows with the polynomial degree, causing each degree of freedom to be coupled to $\mathcal{O}(p)$ other degrees of freedom.
Each such connection results in an additional graph viscosity contribution, rendering the method overly diffusive.
On the other hand, the sparsified operator has a stencil size of $\mathcal{O}(1)$, and as a consequence, we do not observe a degradation of the results with increased $p$.

\begin{figure}
   \centering
   \includegraphics{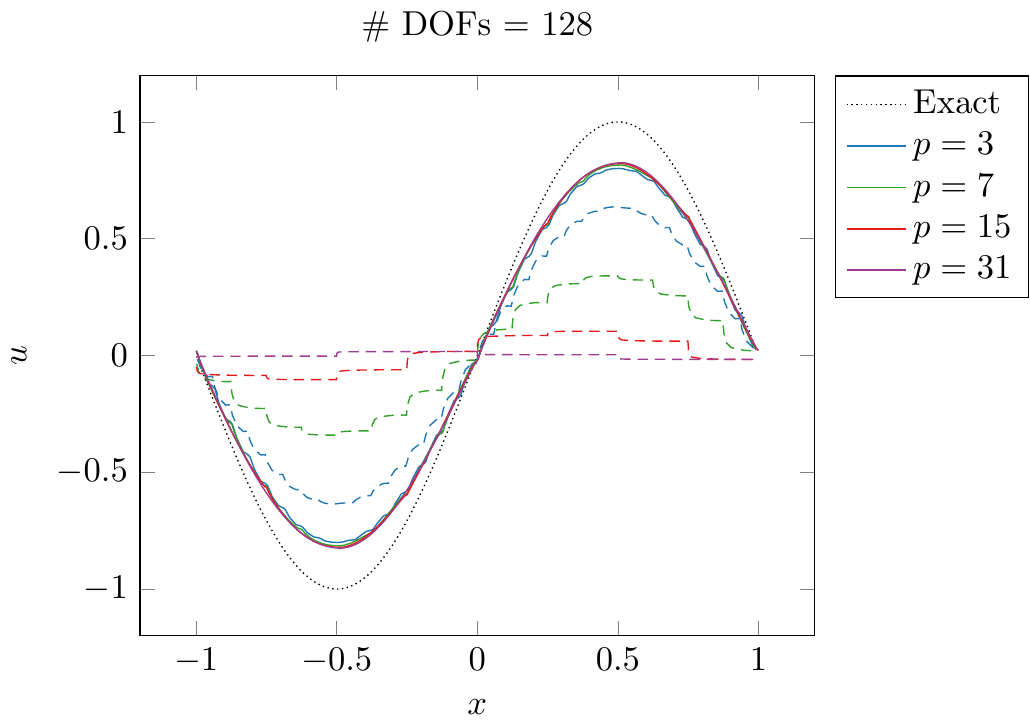}
   \caption{Comparison of sparsified and unsparsified low-order IDP methods applied to the advection equation.
   The polynomial degree is increased and the mesh is simultaneously coarsened to keep the number of degrees of freedom fixed.
   The sparsified methods are shown in solid lines, and the unsparsified methods are shown in dashed lines.
   The exact solution is given by the dotted line.}
   \label{fig:low-order-comparison}
\end{figure}

\section{Flux correction and limiting strategies}
\label{sec:fct}

We now combine the low-order IDP discretization \eqref{eq:dg-idp} with the high-order discretization \eqref{eq:dg-strong}, to obtain a bound-preserving scheme.
To this end, we consider a forward Euler time discretization of both the low-order and high-order approximations (as mentioned above, the extension to high-order time integration is straightforward using SSP methods).
The provisional high-order update for element $K$ is given by
\begin{equation} \label{eq:ho-update}
    \frac{\mathtt{m}_i}{\Delta t} \uu_i^{H,n+1}
    = \frac{\mathtt{m}_i}{\Delta t} \uu_i^n
    - \left(
          \sum_{i=1}^d \D_{i,\Ref} \sum_{j=1}^d \Jinv_{ij} \FF_j^n
        + \sum_{e\in\partial K} \B_{e,\Ref} (\FFhat_e^n - \FF_e^n \cdot\n)
    \right),
\end{equation}
and the low-order IDP update is given by
\begin{equation} \label{eq:lo-update}
    \frac{\mathtt{m}_i}{\Delta t} \uu_i^{L,n+1}
    = \frac{\mathtt{m}_i}{\Delta t} \uu_i^n
    + \left(
          \sum_{i=1}^d \Didp^\tr_{i,\Ref} \sum_{j=1}^d \Jinvidp_{ij} \FF_j^n
        - \sum_{e\in\partial K} \B_{e,\Ref} \FFhat_e^n
    \right)
     - \sum_{j \in \mathcal{E}(i)}\mathtt{d}_{ij} ( \uu_j - \uu_i ).
\end{equation}
We introduce the short-hand notation
\begin{equation}
    \frac{\mathtt{m}_i}{\Delta t} \uu_i^{H,n+1}
    = \frac{\mathtt{m}_i}{\Delta t} \uu_i^n + \mathtt{r}_i^{H,n},
    \hspace{1in}
    \frac{\mathtt{m}_i}{\Delta t} \uu_i^{L,n+1}
    = \frac{\mathtt{m}_i}{\Delta t} \uu_i^{n} + \mathtt{r}_i^{L,n}.
\end{equation}
Note that the provisional high-order solution may not satisfy convex invariants, such as maximum principles, positivity, etc.
However, the low-order solution is guaranteed to be IDP as long as the CFL condition \eqref{eq:cfl-low} is satisfied.
We then define the flux-corrected update by
\begin{equation} \label{eq:fct-update}
    \frac{\mathtt{m}_i}{\Delta t} \uu_i^{n+1}
    = \frac{\mathtt{m}_i}{\Delta t} \uu_i^n
    + \mathtt{r}_i^{L,n}
    + \alpha_i \left(
        \mathtt{r}_i^{H,n} - \mathtt{r}_i^{L,n}
      \right)
\end{equation}
where $0 \leq \alpha_i \leq 1$ is a limiting factor that is yet to be determined.

To ensure conservation of the flux corrected solution, it is important that the correction terms do not impair the following zero-sum property:
\begin{prop} \label{prop:zero-sum-correction}
   Let $\mathcal{I}(K)$ denote the set of indices associated with the element $K\in\T_h$.
   Then,
   \begin{equation}
      \sum_{i\in\mathcal{I}(K)} \left( \mathtt{r}_i^{H,n} - \mathtt{r}_i^{L,n} \right) = 0.
   \end{equation}
\end{prop}
\begin{proof}
Note that the terms associated with the numerical flux $\FFhat$ in both \eqref{eq:ho-update} and \eqref{eq:lo-update} are identical, and therefore cancel in the difference $\mathtt{r}_i^{H,n} - \mathtt{r}_i^{L,n}$.
The remaining terms in both the high-order and low-order residuals are the elemental differentiation matrices, which can be seen to sum to zero elementwise.
\end{proof}

In fact, the above proposition can be extended to consider ``lines'' of degrees of freedom in a dimension-by-dimension fashion.
Recall that each element $K\in\T_h$ contains $(p+1)^d$ nodes, which we can index as $\bm j = (j_1, j_2, \ldots, j_d)$.
Fix a dimension $1 \leq k \leq d$. Then, let $\mathcal{I}_{k,\bm j}(K)$ denote the set of all indices of nodes in the element $K$, with local index $(j_1, j_2, \ldots, i, \ldots, j_d)$, where the $k$th index of $\bm j$ has been replaced by $i$ with $1 \leq i \leq p+1$.
Furthermore, the residuals within a given element are decomposed into contributions corresponding to each coordinate dimension:
\begin{equation}
   \mathtt{r}_i^{H,n} = \sum_{k=1}^d \mathtt{r}_{i,k}^{H,n}, \hspace{1in}
   \mathtt{r}_i^{L,n} = \sum_{k=1}^d \mathtt{r}_{i,k}^{L,n}.
\end{equation}

\begin{prop} \label{prop:line-zero-sum}
   Consider a ``line'' of nodes $\mathcal{I}_{k,\bm j}(K)$ within an element $K\in\T_h$ (for any $1 \leq k \leq d$ and multi-index $\bm j$ as described above). Then
   \begin{equation}
      \sum_{i\in\mathcal{I}_{k, \bm j}(K)} \left( \mathtt{r}_{i,k}^{H,n} - \mathtt{r}_{i,k}^{L,n} \right) = 0.
   \end{equation}
\end{prop}
\begin{proof}
   This property follows from the Kronecker product definition of the operators $\D_{i,\Ref}$ and $\Didp_{i,\Ref}$, e.g.\ as written in \eqref{eq:kronecker}.
\end{proof}

\subsection{Linear constraints for scalar problems}
\label{sec:linear-constraints}

Suppose the governing equation is a scalar conservation law.
In this case, the equation satisfies a local maximum principle, and any interval $[a,b] \subseteq \R$ is a convex invariant set.
Since the low-order method is IDP, we have that if $\uu_j^n \in [a,b]$ for all $j\in\mathcal{N}(i)$, then $\overline{\uu}_{ij}^{n+1} \in [a,b]$.
We wish to enforce a similar property for the high-order limited quantity $\uu_i^{n+1}$.
Specifically, for each $i$, we choose bounds $\uu_i^{\min}$ and $\uu_i^{\max}$ and enforce $\uu_i^{n+1} \in [\uu_i^{\min}, \uu_i^{\max}]$.
Since this constraint is linear (as opposed to nonlinear constraints such as entropy inequalities), the limiting procedure is relatively simple.

\subsubsection{Bounds}
\label{sec:bounds}

In order to choose the limiting factors $\alpha_i$, we must first choose bounds $\uu_i^{\min}$ and $\uu_i^{\max}$ for all $i$.
We can consider several choices of bounds:
\begin{itemize}
   \item $\uu_j^n, \quad j \in \mathcal{N}(i)$,
   \item $\uu_j^{L,n+1}, \quad j \in \mathcal{N}(i)$,
   \item $\overline{\uu}_{ij}^{n+1}, \quad j \in \mathcal{N}(i)$.
\end{itemize}
We have shown earlier in this document that $\uu_j^{L,n+1}$ and $\overline{\uu}_{ij}^{n+1}$ preserve convex invariant sets of the governing equation.
Therefore, these quantities satisfy the discrete maximum principle.
Note that $\uu_j^{L,n+1}$ is a convex combination of $\overline{\uu}_{ij}^{n+1}$, and hence represents a more restrictive bound.
To define $\uu_i^{\min}$ and $\uu_i^{\max}$ we may take the minimum and maximum of some combination of these quantities.
For the remainder of this section, we choose the bounds naturally satisfied by the low-order discretization:
\begin{equation}
  \uu_i^{\min} = \min_{j\in\mathcal{N}(i)} \overline{\uu}_{ij}^{n+1},
  \qquad
  \uu_i^{\max} = \max_{j\in\mathcal{N}(i)} \overline{\uu}_{ij}^{n+1}.
\end{equation}

\subsubsection{Elementwise limiting for linear constraints}

A Zalesak-type limiter can be used to compute a provisional limiting factor $\tilde{\alpha}_i$ for each $i$, cf.~
\cite{Lohmann2017,Zalesak1979}.
Then, $\alpha_i$ can be defined as the minimum of all $\tilde{\alpha}_i$ over the element $K$:
\begin{equation}
   \alpha_i = \min_{j \in \mathcal{I}(K)} \tilde{\alpha}_j \qquad \text{for all $i \in \mathcal I(K)$}.
\end{equation}
Given this definition, $\alpha_i$ is constant on each element, and therefore Proposition \ref{prop:zero-sum-correction} implies that the resulting FCT method is conservative.
However, choosing $\alpha_i$ to be constant over an entire element may be overly pessimistic if the polynomial degree is high, and therefore we consider also subcell limiting.

\subsubsection{Subcell limiting for linear constraints}
\label{sec:subcell-linear}

In order to improve the resolution of the flux-corrected solution, we consider subcell limiting, where the correction factors $\alpha_i$ are allowed to vary within each element.
The idea of the subcell limiting is closely related to that presented in \cite[Section 4.5]{Lohmann2017}.
In that work, a one-dimensional subcell limiting algorithm was described, and the extension to multiple dimensions was proposed as a minimization problem.
Instead of solving a minimization problem, Proposition \ref{prop:line-zero-sum} allows for the use of the one-dimensional subcell limiting procedure along lines of nodes, in a dimension-by-dimension fashion.
For simplicity of notation, let $\mathtt{r}_{i,k}$ denote the antidiffusive flux at the $i$th node in the $k$th dimension: $\mathtt{r}_{i,k} = \mathtt{r}_{i,k}^{H,n} - \mathtt{r}_{i,k}^{L,n}$.

We consider a decomposition of a given element into subcells, as illustrated in Figure \ref{fig:subcell}.
Note that the subcell boundaries are placed in between nodes, such that every node can be considered as a ``subcell-centered value''.
We consider the set of all ``subcell faces,'' which are the set of all \emph{interior} subcell faces (the thin blue lines in the figure).
For each subcell face, we assign an antidiffusive flux that is obtained by summing the nodal antidiffusive fluxes lying on one side of the face.

\begin{figure}
   \centering
   \includegraphics{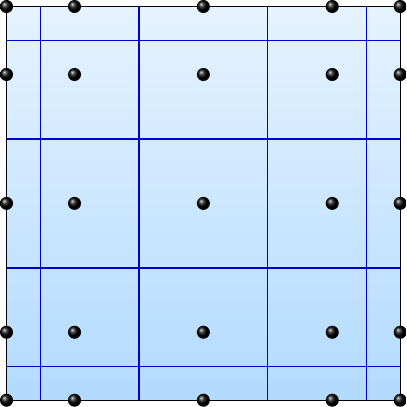}
   \hspace{1cm}
   \includegraphics{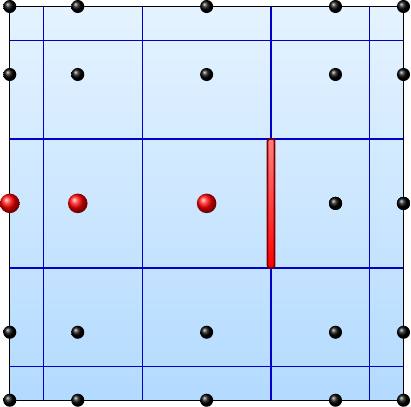}
   \hspace{1cm}
   \includegraphics{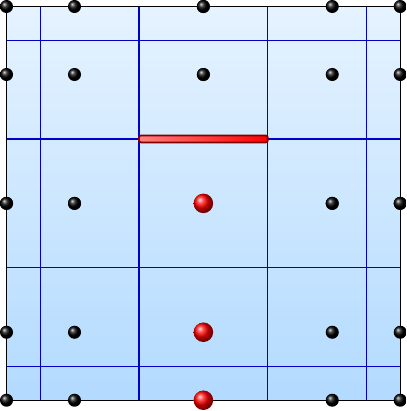}
   \caption{Left: decomposition of elements into subcells, where nodal values are considered to be ``subcell-centered'' values.
   Center and right: definition of a subcell face (indicated by a thick red line), with contributing nodal values (indicated by red nodes).}
   \label{fig:subcell}
\end{figure}

To introduce notation, we refer to the subcell faces by a pair $(m, k)$, where $1 \leq k \leq d$.
The index $k$ indicates that the face $(m,k)$ is normal to the $k$th unit vector.
For a given face $(m,k)$ (indicated by a thick red line in Figure \ref{fig:subcell}), let $\mathcal{M}(m,k)$ denote set of indices of nodes lying on one side of the face (indicated by red nodes in Figure \ref{fig:subcell}).
We then define the subcell face flux associated with the face $(m,k)$ by summing the directional nodal fluxes over the set $\mathcal{M}(m,k)$:
\begin{equation}
  \overline{\mathtt{r}}_{m,k} = \sum_{i \in \mathcal{M}(m,k)} \mathtt{r}_{i,k}
\end{equation}

Fix a node $i$ and direction $k$, and let $(m_i^+,k)$ and $(m_i^-,k)$ denote the subcell faces adjacent to $i$ in the $k$th direction.
Note that $\mathcal{M}(m_i^+,k) = \mathcal{M}(m_i^-,k) \cup \{ i \}$, and therefore, by Proposition \ref{prop:line-zero-sum}, we can write the nodal antidiffusive flux as the difference of adjacent subcell fluxes:
\begin{equation}
   \mathtt{r}_{i,k} = \overline{\mathtt{r}}_{m_i^+,k} - \overline{\mathtt{r}}_{m_i^-,k}.
\end{equation}

For each node $i$, we introduce a nodal provisional limiting coefficient $\tilde{\alpha}_i$, obtained by limiting the sums of positive and negative parts of the adjacent subcell residuals, $\overline{\mathtt{r}}_{m_i^\pm,k}$ (for $1 \leq k \leq d$) according to a Zalesak-type methodology.
Once the nodal provisional limiting coefficients are computed, we define subcell face limiting coefficients $\alpha_{m,k}$.
The quantity $\alpha_{m,k}$ is simply given as the minimum of the two nodal provisional limiting coefficients corresponding to nodes adjacent to the subcell face $(m,k)$:
\begin{equation}
   \alpha_{m,k} = \min \{ \tilde{\alpha}_{i_1}, \tilde{\alpha}_{i_2} \},
\end{equation}
where $i_1$ and $i_2$ are nodes adjacent to $(m,k)$.
The subcell limiting coefficients are used to define the subcell flux corrected solution:
\begin{equation} \label{eq:subcell-fct}
    \frac{\mathtt{m}_i}{\Delta t} \uu_i^{n+1}
    = \frac{\mathtt{m}_i}{\Delta t} \uu_i^n
    + \mathtt{r}_i^{L,n}
    + \sum_{k=1}^d \left(
          \alpha_{m_i^+,k} \overline{\mathtt{r}}_{m_i^+,k}
        - \alpha_{m_i^-,k} \overline{\mathtt{r}}_{m_i^-,k}
      \right).
\end{equation}
Just as in the case of elementwise limiting, in order for the flux correction procedure to be conservative, we require that the limiting corrections sum to zero on each element.
This property is summarized in the following proposition:
\begin{prop} \label{prop:subcell-conservation}
   Fix an element $K$, and let $\mathcal{I}(K)$ denote the set of all nodal indices in element $K$.
   Let $\alpha_{m,k}$ and $\overline{\mathtt{r}}_{m,k}$ be defined as above.
   Then,
   \begin{equation} \label{eq:subcell-zero-sum}
      \sum_{i \in \mathcal{I}(K)} \sum_{k=1}^d \left(
          \alpha_{m_i^+,k} \overline{\mathtt{r}}_{m_i^+,k}
        - \alpha_{m_i^-,k} \overline{\mathtt{r}}_{m_i^-,k}
      \right) = 0.
   \end{equation}
\end{prop}
\begin{proof}
   Each subcell face $(m,k)$ is always adjacent to exactly two nodes, and note that the terms $\overline{\mathtt{r}}_{m_i^+,k}$ and $\overline{\mathtt{r}}_{m_i^-,k}$ appear in \eqref{eq:subcell-zero-sum} with opposite signs.
   Additionally, for any subcell face $(m,k)$ the term $\overline{\mathtt{r}}_{m,k}$ appears with coefficient $\alpha_{m,k}$.
   Therefore, summing over all nodes $i$, these terms cancel.
\end{proof}

Since $\alpha_{m_i^\pm,k} \leq \tilde{\alpha}_i$, and the nodal provisional limiting coefficient $\tilde{\alpha}_i$ is obtained by limiting the positive and negative parts of $\overline{\mathtt{r}}_{m_i^\pm,k}$ separately, we can see that the update given by \eqref{eq:subcell-fct} will satisfy the desired bounds.

\subsubsection{Comparison of elementwise and subcell limiting: advection equation}
\label{sec:subcell-comparison-advection}

We consider a simple one-dimensional test to compare the effectiveness of the subcell and elementwise limiting techniques.
We consider the advection equation $u_t + u_x = 0$ on the domain $[-1,1]$, with periodic boundary conditions, and initial conditions given by two square waves.
The discontinuities in the initial condition are aligned with the mesh.
We fix the number of degrees of freedom to be 320, and use polynomial degrees $p=0,1,3,7$.
We integrate until a final time of $t=2$, at which point the solution and initial condition coincide.
The results are shown in Figure \ref{fig:subcell-comparison}.
From these results, we notice that using the elementwise limiting strategy, increasing the polynomial degree while simultaneously coarsening the mesh does not improve solution quality beyond $p=1$.
This is intuitively the case because the limiting coefficients lack subcell resolution.
On the other hand, the subcell limiting technique results in increased accuracy as the polynomial degree is increased and as the mesh is coarsened.

\begin{figure}
   \centering
   \includegraphics{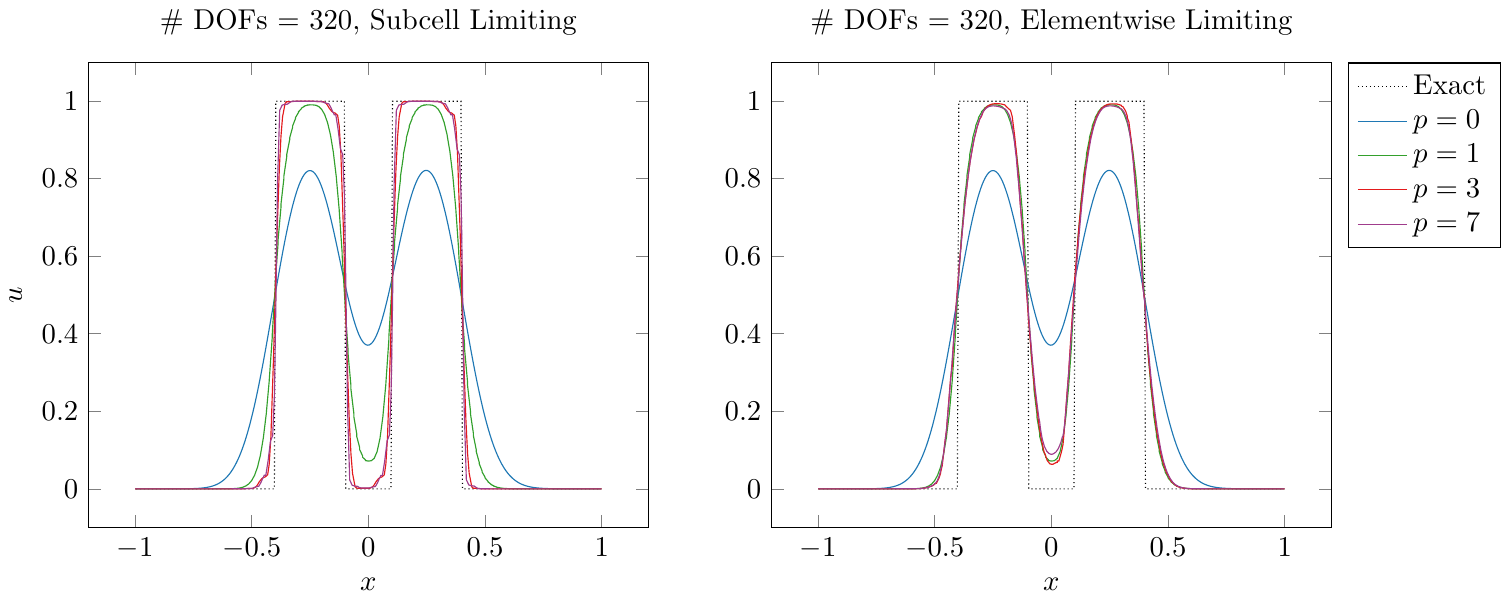}
   \caption{Comparison of subcell and elementwise limiting strategies, with fixed number of degrees of freedom, varying the polynomial degree.}
   \label{fig:subcell-comparison}
\end{figure}

Additionally, we consider the two-dimensional solid body rotation test case.
This test case is described in detail in Section \ref{sec:solid-body}.
We compare the subcell and elementwise limiting strategies using a coarse mesh with $25\times25$ elements, and polynomial degree $p=3$.
We integrate in time for one complete revolution, using a time step that is half of the CFL condition given by \eqref{eq:cfl-low}.
The results are shown in Figure \ref{fig:solid-body-subcell}.
It is clear that the subcell limiting strategy results in less diffusive results and better resolution of the features when compared with the elementwise limiting strategy.

\begin{figure}
   \centering
   \includegraphics[width=0.3\linewidth]{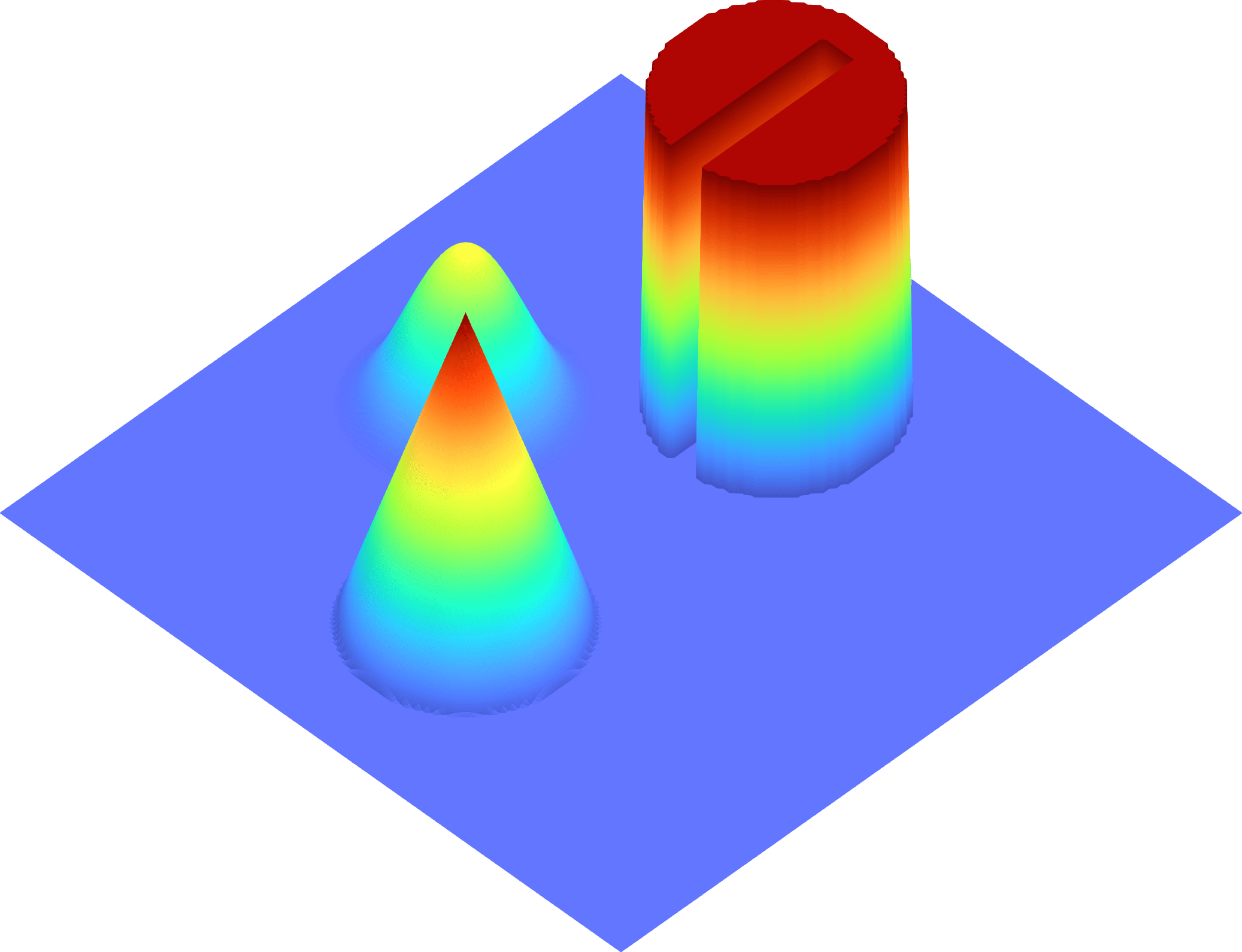}
   \hfill
   \includegraphics[width=0.3\linewidth]{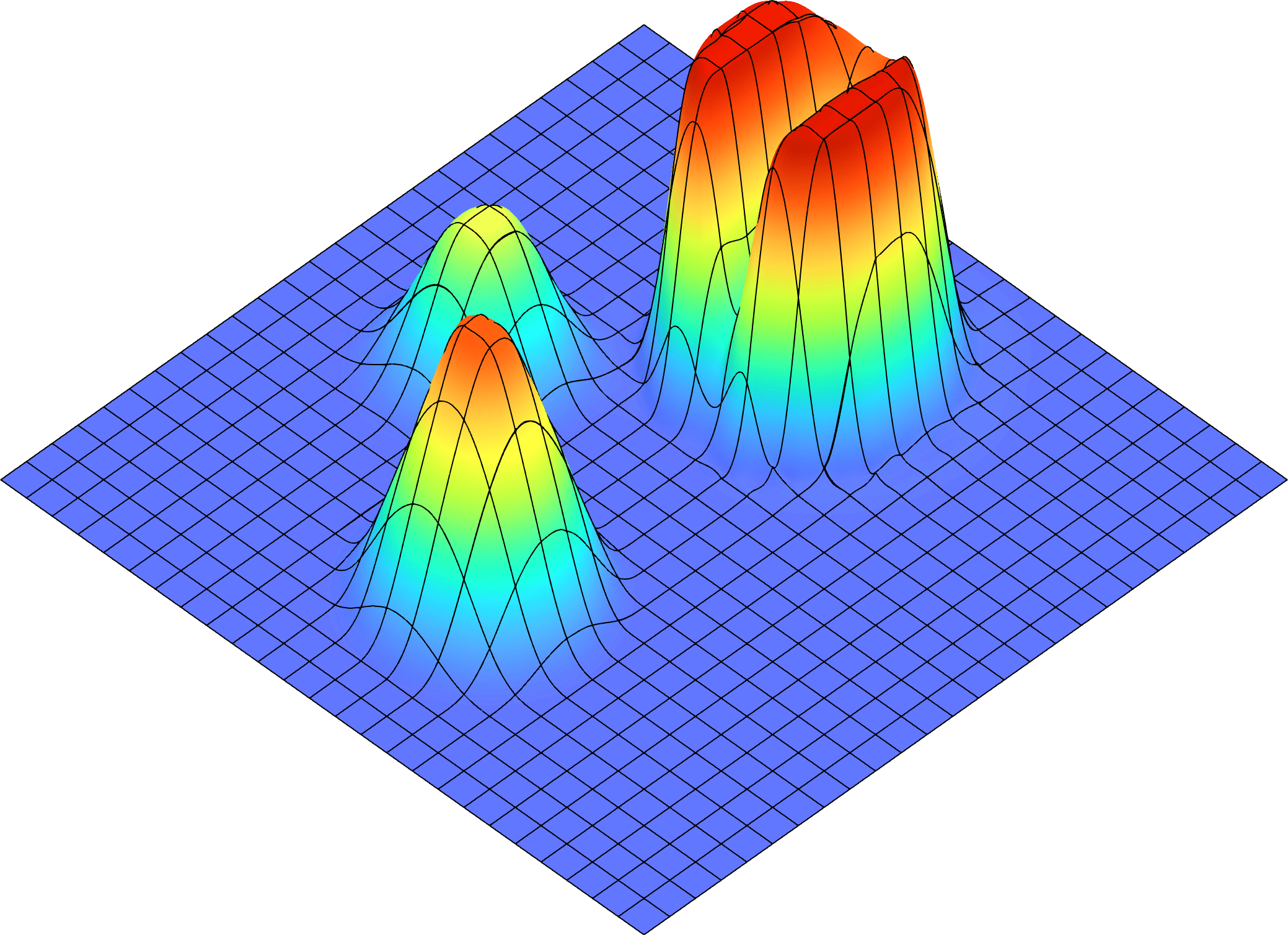}
   \hfill
   \includegraphics[width=0.3\linewidth]{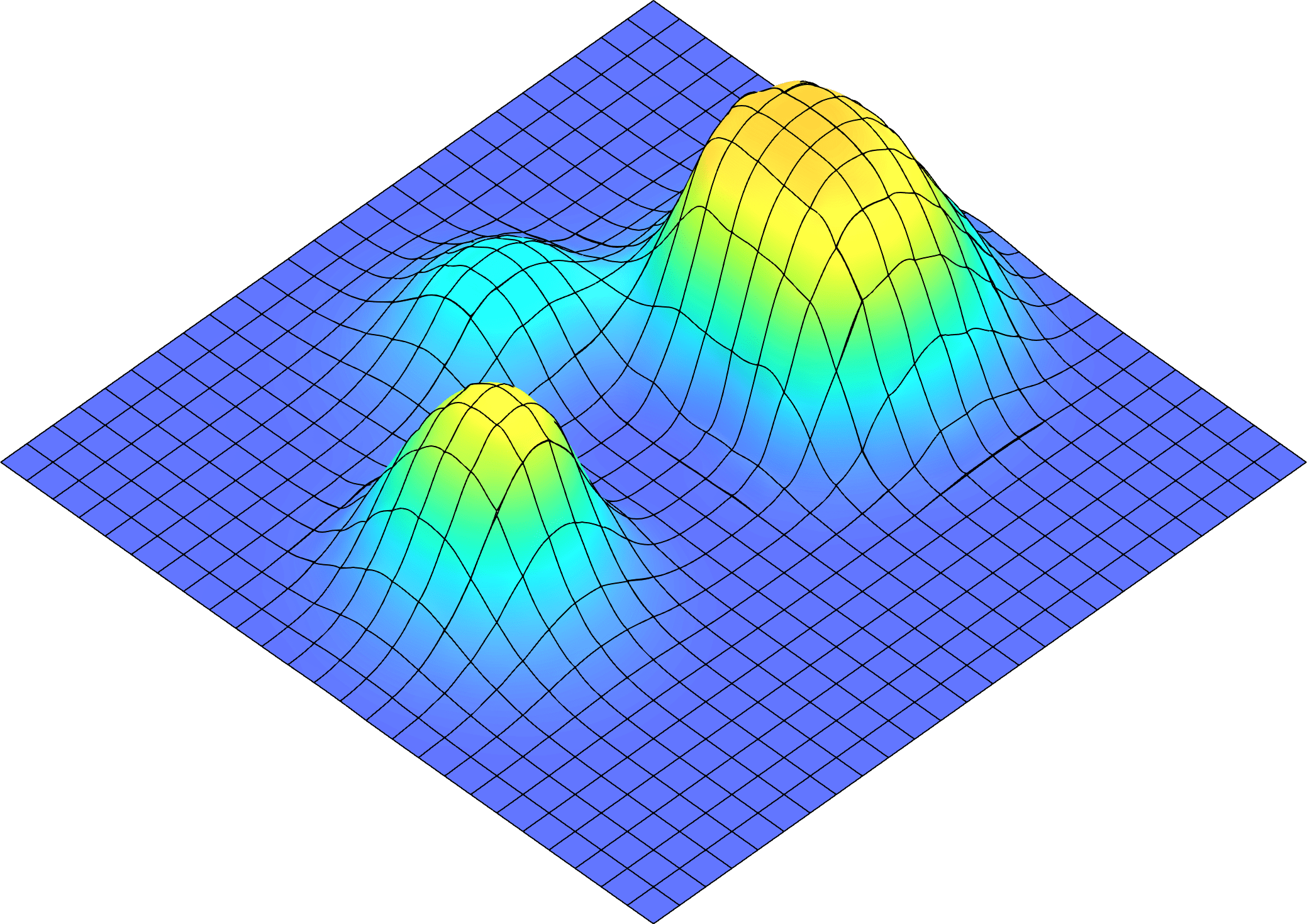}
   \caption{Comparison of subcell and elementwise limiting strategies for the solid body rotation test on a $25 \times 25$ mesh with polynomial degree $p=3$. Left: exact solution. Center: subcell limiting. Right: elementwise limiting.}
   \label{fig:solid-body-subcell}
\end{figure}

\subsection{Convex invariants for hyperbolic systems}
\label{sec:convex-invariants}

We now consider the more general case, where \eqref{eq:cons-law} represents a system of conservation laws, and we wish to enforce several \emph{convex constraints}.
In this case, the Zalesak-type limiter which we used for linear constraints is no longer sufficient.
Instead, we make use of the \emph{convex limiting} methodology developed in \cite{Guermond2018,Guermond2019}.

We have shown that for any convex invariant set $\mathcal{A}$ of \eqref{eq:cons-law} such that $\uu_i^n \in \mathcal{A}$ for all $i$, we have $\uu_i^{L,n+1} \in \mathcal{A}$ for all $i$.
We choose some finite subcollection of such sets, denoted $\mathcal{A}_\ell$.
We now describe a limiting procedure to ensure that if $\uu_i^n \in \mathcal{A}_\ell$ for all $i$ and all $\ell$, then $\uu_i^{n+1} \in \mathcal{A}_\ell$ for all $i$ and all $\ell$.
As before, we first begin by describing an elementwise limiting procedure, and then go on to develop a dimension-by-dimension subcell limiting procedure.

\subsubsection{Elementwise limiting for convex constraints}

The elementwise limiting strategy for convex constraints is almost identical to that for linear constraints.
For a given element $K$, we assign to each node $i \in \mathcal{E}(K)$ a provisional limiting coefficient $\tilde{\alpha}_i$ as follows.
Let $\tilde{\alpha}_i \in [0,1]$ be the largest value such that the limited nodal value $\tilde{\uu}_i^{n+1}$  given by
\begin{equation} \label{eq:convex-provisional-elementwise}
   \frac{\mathtt{m}_i}{\Delta t} \tilde{\uu}_i^{n+1}
   = \frac{\mathtt{m}_i}{\Delta t} \uu_i^n
   + \mathtt{r}_i^{L,n}
   + \tilde{\alpha}_i \left(
       \mathtt{r}_i^{H,n} - \mathtt{r}_i^{L,n}
     \right)
\end{equation}
belongs to each of the convex invariants, i.e.\ $\tilde{\uu}_i^{n+1} \in  \mathcal{A}_\ell$ for all $\ell$.
We then define the elementwise limiting coefficients by
\begin{equation}
   \alpha_i = \min_{j\in\mathcal{E}(K)} \tilde{\alpha}_j.
\end{equation}
Therefore, $\alpha_i \in [0, \tilde{\alpha_i}]$ for all $i$.
Since the sets $\mathcal{A}_\ell$ are convex, and $\uu_i^{L,n+1} \in \mathcal{A}_\ell$, we conclude that the flux-corrected nodal values $\uu_i^{n+1}$ defined by \eqref{eq:fct-update} also satisfy $\uu_i^{n+1} \in \mathcal{A}_\ell$ for all $\ell$.

\subsubsection{Subcell limiting for convex constraints}

As before, we make use of the dimension-by-dimension decomposition of residuals described in Section \ref{sec:subcell-linear}.
We begin by fixing an element $K$.
For each node $i \in \mathcal{E}(K)$, we define a provisional limiting coefficient $\tilde{\alpha}_i$,
according to the following procedure.
Recall that $\mathtt{r}_{i} = \sum_{k=1}^d \mathtt{r}_{i,k} = \sum_{k=1}^d \left( \overline{\mathtt{r}}_{m_i^+,k} - \overline{\mathtt{r}}_{m_i^-,k} \right)$.
Let $\gamma = 2d$, and then choose $\tilde{\alpha_i} \in [0,1]$ to be the largest value such that
the provisional updates given by
\begin{equation} \label{eq:convex-provisional}
   \frac{\mathtt{m}_i}{\Delta t} \tilde{\uu}_i^{n+1}
   = \frac{\mathtt{m}_i}{\Delta t} \uu_i^n
   + \mathtt{r}_i^{L,n}
   + \gamma \tilde{\alpha}_i
       \left( \pm \overline{\mathtt{r}}_{m_i^\pm,k} \right)
\end{equation}
for all $1 \leq d \leq k$ satisfy $\tilde{\uu}_i^{n+1} \in \mathcal{A}_\ell$ for all $\ell$.
As in the case of elementwise limiting, we note that since the sets $\mathcal{A}_\ell$ are convex, the same will hold for any limiting coefficient in the interval $[0,\tilde{\alpha}_i]$.

Let $1 \leq k \leq d$ be a given dimension, and let $(m,k)$ denote the index of a subcell face.
As in the case of linear constraints, let $\alpha_{m,k}$ be defined by $\alpha_{m,k} = \min\{\tilde{\alpha}_{i_1}, \tilde{\alpha}_{i_2} \},$ where $i_1$ and $i_2$ are nodes adjacent to the subcell face $(m,k)$.
The flux correction nodal values are given, as before, by equation \eqref{eq:subcell-fct}, which we write in the slightly modified form
\begin{equation} \label{eq:convex-fct}
    \uu_i^{n+1}
    = \uu_i^{L,n+1}
    + \frac{\Delta t}{\mathtt{m}_i}\sum_{k=1}^d \left(
          \alpha_{m_i^+,k} \overline{\mathtt{r}}_{m_i^+,k}
        - \alpha_{m_i^-,k} \overline{\mathtt{r}}_{m_i^-,k}
      \right).
\end{equation}
This definition gives rise to the following property.
\begin{prop}
   Let $\uu_i^{n+1}$ be defined by \eqref{eq:convex-fct}, where the limiting coefficients $\alpha_{m_i^\pm,k}$ are given by the procedure described above.
   Then, the flux-corrected nodal values $\uu_u^{n+1}$ satisfy $\uu_i^{n+1} \in \mathcal{A}_\ell$ for all $\ell$.
\end{prop}
\begin{proof}
We rewrite \eqref{eq:convex-fct} as the convex combination
\begin{equation} \label{eq:fct-convex-combination}
   \begin{aligned}
    \uu_i^{n+1}
    &= \sum_{k=1}^d \left(
       \frac{1}{2d} \uu_i^{L,n+1} + \frac{\Delta t}{\mathtt{m}_i} \alpha_{m_i^+,k} \overline{\mathtt{r}}_{m_i^+,k}
       + \frac{1}{2d} \uu_i^{L,n+1} - \frac{\Delta t}{\mathtt{m}_i} \alpha_{m_i^-,k} \overline{\mathtt{r}}_{m_i^-,k}
    \right) \\
    &= \frac{1}{\gamma} \sum_{k=1}^d \left(
       \uu_i^{L,n+1} + \frac{\Delta t}{\mathtt{m}_i} \gamma \alpha_{m_i^+,k} \overline{\mathtt{r}}_{m_i^+,k}
       \right)
       + \frac{1}{\gamma} \sum_{k=1}^d \left(
       \uu_i^{L,n+1} - \frac{\Delta t}{\mathtt{m}_i} \gamma \alpha_{m_i^-,k} \overline{\mathtt{r}}_{m_i^-,k}
    \right).
    \end{aligned}
\end{equation}
Note that $\alpha_{m_i^\pm,k} \leq \tilde{\alpha}_i$, and therefore, by \eqref{eq:convex-provisional}, each term in the convex combination in \eqref{eq:fct-convex-combination} lies within the convex invariant sets $\mathcal{A}_\ell$.
Therefore, the flux corrected nodal values $\uu_i^{n+1}$ defined by \eqref{eq:convex-fct} satisfy $\uu_i^{n+1}\in\mathcal{A}_\ell$ for all $\ell$.
\end{proof}

Additionally, note that Proposition \eqref{prop:subcell-conservation}, shown for the case of linear constraints, also applies to the subcell convex limiting strategy described here.
Therefore, the subcell convex limiting strategy is also conservative.

\subsection{Application: Euler equations}

We now consider the application of the above techniques to the compressible Euler equations.
We use conserved variables $\bm u = (\rho, \rho \bm v, \rho E)$, where $\rho$ is the density, $\bm v \in \R^d$ is the velocity, and $E$ is the total energy per unit mass.
For an ideal gas, the pressure $p$ is defined by the equation of state
\begin{equation} \label{eq:euler-eos}
  p = (\gamma - 1)\rho\left( E - \|\bm{v}\|^2/2\right),
\end{equation}
where $\gamma$ is the ratio of specific heats.
In this work, we take $\gamma = 1.4$.
The governing equations are given by \eqref{eq:cons-law} with flux function
\begin{equation} \label{eq:euler-flux}
   \bm F = \left( \begin{array}{c}
    \rho \bm{v} \\
    \rho \bm{v} \otimes \bm{v}^T + p I \\
    \rho H \bm{v}
  \end{array} \right),
\end{equation}
where $I$ is the $d \times d$ identity matrix, and $H = E + p / \rho$ is the stagnation enthalpy.
The specific internal energy $e$ is defined by $e = E - \frac{1}{2}\bm v^2$ (using the notation $\bm v^2 = \| \bm v \|_{\ell^2}^2$), and the specific entropy $s$ is given by $s = \log\left( e^{\frac{1}{\gamma-1}} \rho^{-1} \right)$.

The set $\mathcal{A}(r)$, defined by
\begin{equation}
   \mathcal{A}(r) = \left\{
      (\rho, \rho \bm v, \rho E) :
      \rho > 0, e > 0, s \geq r
   \right\}
\end{equation}
for any $r \geq 0$ is a convex invariant set for the Euler equations, and is an invariant domain for the Lax-Friedrichs method \cite{Guermond2016, Guermond2018, Frid2001}.
Note that, by Proposition \ref{prop:bar-states}, the bar states $\overline{\uu}_{ij}^{n+1}$ defined by \eqref{eq:bar-states} have positive density, internal energy, and satisfy the minimum principle on specific entropy.
Therefore, if $\uu_i^n \in \mathcal{A}(r)$, the low order solutions $\uu_i^{L,n+1}$ defined by \eqref{eq:idp-fe} satisfy $\uu_i^{L,n+1} \in \mathcal{A}(r)$.

We will use the FCT-based convex limiting techniques described above to ensure that the target solution $\uu_i^{n+1}$ also satisfies $\uu_i^{n+1} \in \mathcal{A}(r)$.
This is achieved through a two-part limiting process using a procedure similar to that described in \cite{Guermond2016}.
First, we limit the density using the Zalesak-type limiter from Section \ref{sec:linear-constraints}, enforcing bounds given by
\begin{equation}
   \rho_i^{\min} = \min_{j \in \mathcal{N}(i)} \overline{\rho}_{ij}^{n+1},
   \qquad\qquad
   \rho_i^{\max} = \max_{j \in \mathcal{N}(i)} \overline{\rho}_{ij}^{n+1}.
\end{equation}
Then, we enforce a minimum principle on the specific entropy using the convex limiting procedure from Section \ref{sec:convex-invariants}.
The lower bound for specific entropy is given by
\begin{equation}
   s_i^{\min} = \min_{j \in \mathcal{N}(i)} s_j^n,
\end{equation}
where $s_j^n = s(\bm u_j^n)$.
Determining the provisional nodal limiting coefficients (in \eqref{eq:convex-provisional-elementwise} or \eqref{eq:convex-provisional}) requires performing a line search.
By virtue of the convexity of the specific entropy $s$, this line search can be performed efficiently using Newton's method.
The minimum principle on specific entropy ensures that the internal energy is positive.

\subsection{Subcell resolution smoothness indicator}
\label{sec:smoothness}

FCT methods often suffer from a phenomenon known as \emph{peak clipping} \cite{Kuzmin2012,Book2012}.
Because the limiting techniques described above result in methods that are local extremum diminishing, smooth extrema tend to decrease in amplitude slightly with each time step.
Total variation diminishing (TVD) schemes are provably at most first-order accurate at smooth extrema \cite{Osher1984,Zhang2011}.
Smoothness indicators making use of second derivative information are a typical way to alleviate this difficulty \cite{Kuzmin2013,Lohmann2017,Hajduk2020}.
These smoothness indicators often use estimates of the second derivatives to determine regions where the solution is smooth.
In this work, we make use of a slightly different approach, based on the idea of artificial viscosity subcell shock capturing for discontinuous Galerkin methods \cite{Persson2006,Persson2013}.

Consider the solution $\bm u_h$ restricted to a single element $K\in\T_h$, denoted $\bm u_K = \bm u_h|_K$.
We represent $\bm u_K$ in terms of a modal (Legendre) basis, and define a truncated solution $\hat{\bm u}_K$, which is obtained from $\bm u_K$ by setting to zero the coefficients associated with highest-degree basis functions in any variable.
In other words, $\hat{\bm u}_K \in \Q_{p-1}(K)$.
The smoothness indicator is determined by how well $\hat{\bm u}_K$ approximates $\bm u_K$, based on the observation that the high modes of functions well-resolved on the mesh will quickly decay.
We define the smoothness indicator
\begin{equation}
   s_K = \log_{10} \left(
      \frac{\| \bm u_K - \hat{\bm u}_K \|_{L^2}^2}{ \| \bm u_K \|_{L^2}^2}
   \right).
\end{equation}
The indicator $s_K$ is used to define a smoothness factor $\varepsilon_K \in [0,1]$ by
\begin{equation}
   \varepsilon_K =
   \begin{cases}
      0, & \text{ if $s_K < s_0 - \kappa$,} \\
      \frac{1}{2} -  \frac{1}{2}\sin\left(\pi(s_K-s_0)/(2\kappa)\right),
         & \text{ if $s_0 - \kappa \leq s_K \leq s_0 + \kappa$,} \\
      1, & \text{ if $s_K > s_0 + \kappa$.}
   \end{cases}
\end{equation}
Here, $s_0$ and $\kappa$ are user-defined parameters.
In the context of artificial viscosity shock capturing, the choice of these parameters can have a significant impact on the quality of the method \cite{Klockner2011,Pazner2019}.
In this work, we choose $s_0 \sim \log_{10} (p^{-4})$ and $\kappa = 1$, which have be found to be effective choices for $p > 1$.
The factor $\varepsilon_K$ is used to relax the local bounds described in Section \ref{sec:bounds}.
For any $i \in \mathcal{E}(K)$, define relaxed bounds by
\begin{equation} \label{eq:relaxed-bounds}
   \hat{\uu}_i^{\min} = \varepsilon_K \uu_i^{\min} + (1 - \varepsilon_K) \mathtt{g}^{\min},
   \qquad\qquad
   \hat{\uu}_i^{\max} = \varepsilon_K \uu_i^{\max} + (1 - \varepsilon_K) \mathtt{g}^{\max},
\end{equation}
where $\mathtt{g}^{\min}$ and $\mathtt{g}^{\max}$ are relaxed bounds for the problem.
For example, these bounds may be given by the minimum and maximum values of the initial condition.
The relaxed bounds given by \eqref{eq:relaxed-bounds} have the property that they are equal to the local bounds in regions where the solution is rough or under-resolved, and are equal to the global problem bounds in regions where the solution is smooth.

In Figure \ref{fig:smoothness-comparison-1d}, we illustrate the effect of the smoothness indicator on two 1D advection problems.
For both problems, the global bounds are given by the global maximum and minimum of the solution.
The first problem has a smooth solution, and therefore the smoothness indicator will cause only global bounds to be enforced.
For this problem we use a very coarse mesh with only 4 elements and degree $p=7$ polynomials.
This greatly reduces the peak clipping effect, and results in a highly accurate solution.

The second test case consists of a discontinuous initial condition.
For this test case, we use $p=3$ on a mesh with 80 elements.
We note that enforcing only the global bounds in the vicinity of the shorter peak would also for the introduction of oscillations and new local maxima.
However, cells containing the discontinuity are detected by the smoothness indicator, causing local bounds to be enforced in these regions.
As a consequence, noticeable oscillations are not introduced when using the smoothness indicator in this case.

\begin{figure}
   \includegraphics{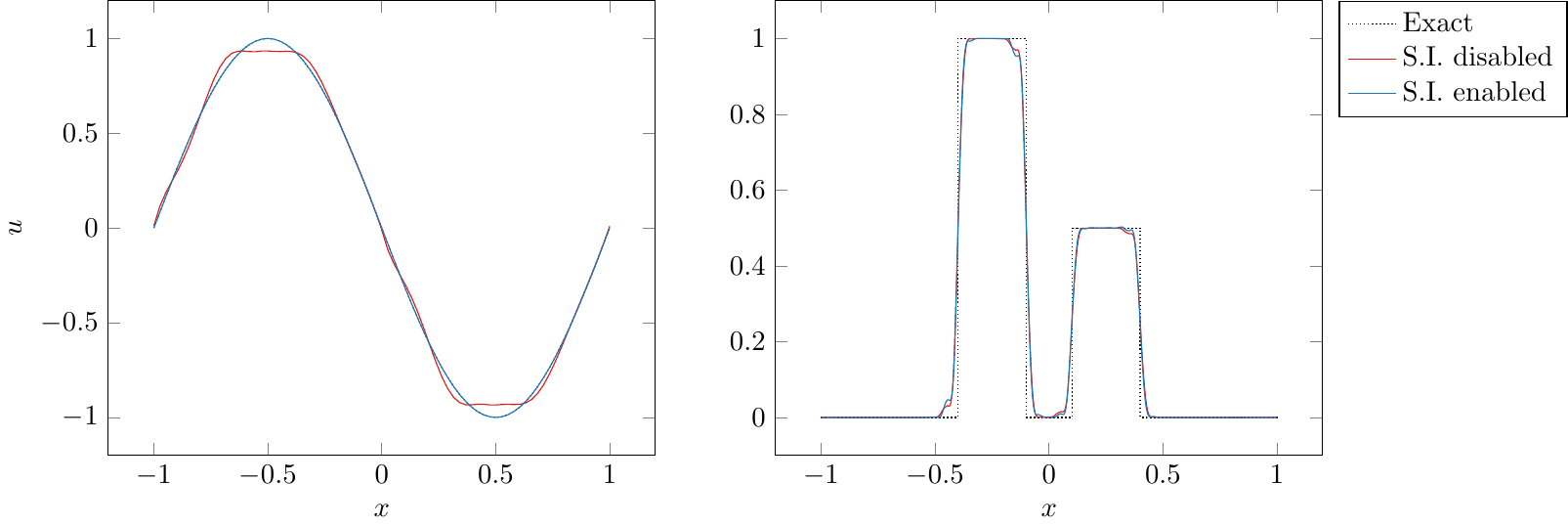}
   \caption{Comparison of two 1D advection problems, with and without smoothness indicators.
   For problems with smooth solutions (left plot), we see that using the smoothness indicator greatly reduces the effect of peak clipping.
   For problems with discontinuities (right plot), the the use of the smoothness indicator still results in a bounds preserving solution, and does not introduce noticeable oscillations.}
   \label{fig:smoothness-comparison-1d}
\end{figure}

Additionally, we study the effects of the smoothness indicator on the solid body rotation test.
As in Section \ref{sec:subcell-comparison-advection}, we use a coarse $25\times25$ mesh with polynomial degree $p=3$.
We use the subcell limiting technique described above, and compare the results with the smoothness indicator disabled and enabled.
The results are shown in Figure \ref{fig:solid-body-smoothness}.
We see that enabling the smoothness indicator results in sharper resolution of features such as the peak of the cone and the edges of the slotted cylinder.
The solution quality is not degraded by spurious oscillations, and the $L^1$ error is about 25\% smaller for this example.

\begin{figure}
   \centering
   \begin{minipage}{0.3\linewidth}
      \centering
      Smoothness Indicator Disabled\\[12pt]
      \includegraphics[width=\linewidth]{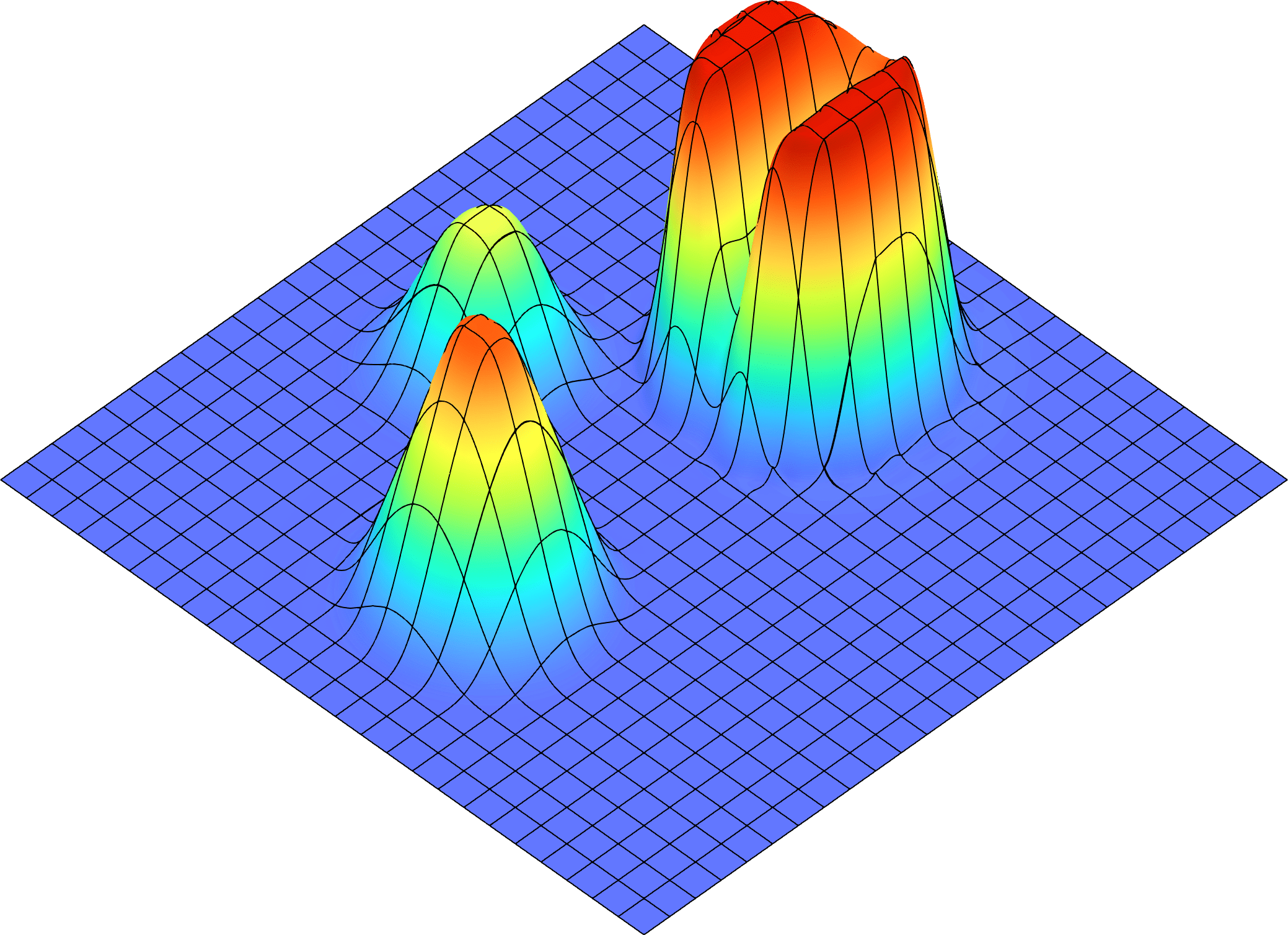}\\[12pt]
      \pbox{\linewidth}{ \small
      $[\min u_h, \max u_h] = [0, 0.95]$\\
      $\|u - u_h\|_{L^1} = 2.6\times10^{-2}$}
   \end{minipage}
   \hspace{2cm}
   \begin{minipage}{0.3\linewidth}
      \centering
      Smoothness Indicator Enabled\\[12pt]
      \includegraphics[width=\linewidth]{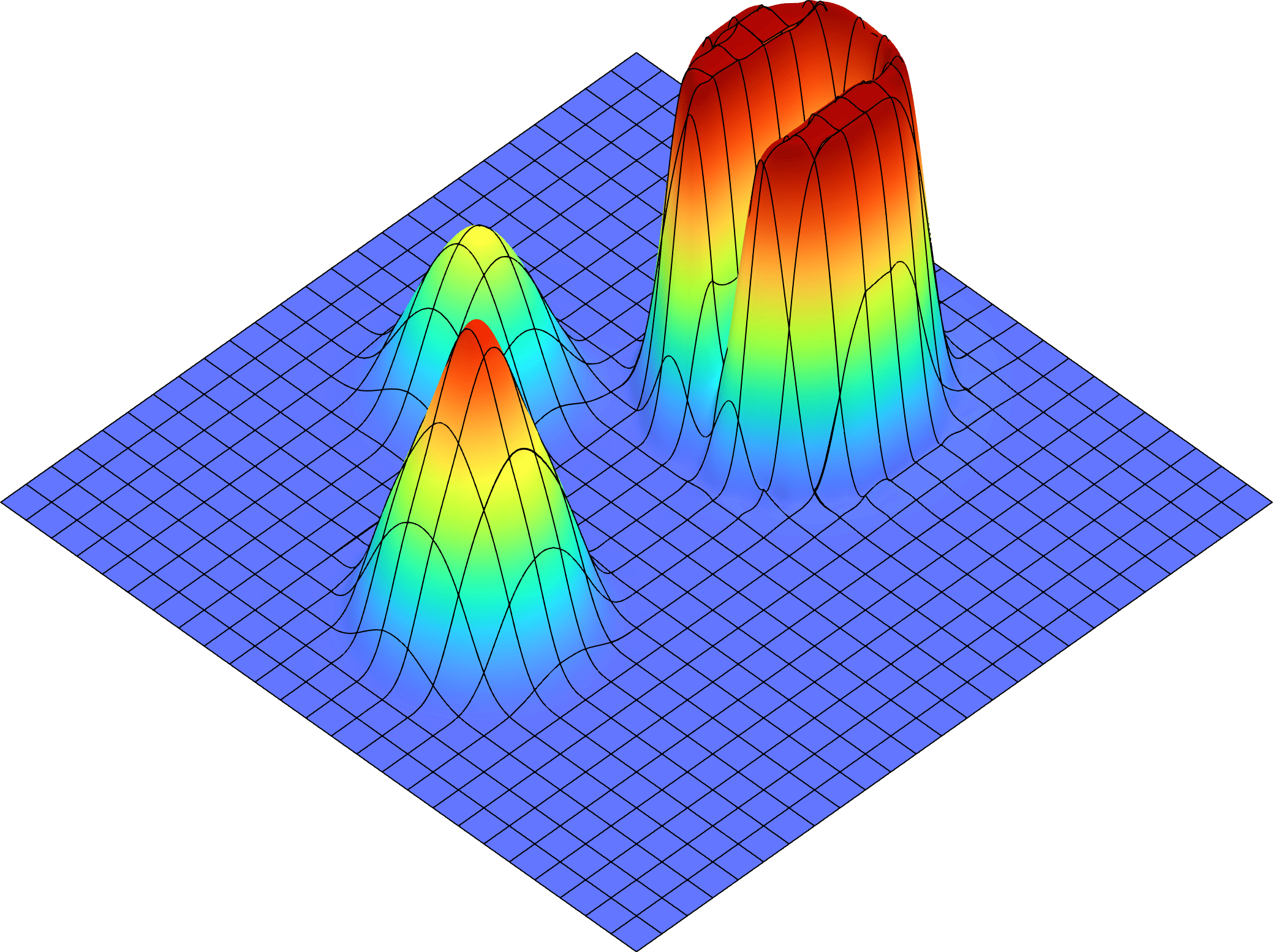}\\[12pt]
      \pbox{\linewidth}{ \small
      $[\min u_h, \max u_h] = [0, 1]$\\
      $\|u - u_h\|_{L^1} = 1.9\times10^{-2}$}
   \end{minipage}

   \caption{Solid body rotation test with smoothness indicator disabled (left) and enabled (right).
   Enabling the smoothness indicator results in sharper resolution of the features and gives a less dissipative solution.}
   \label{fig:solid-body-smoothness}
\end{figure}

\section{Numerical examples}
\label{sec:results}

The method was implemented in the MFEM finite element framework~\cite{Anderson2019}, and is tested on a variety of benchmark problems, including scalar problems and hyperbolic systems, in 1D and 2D.
For these test cases, unless stated otherwise, we integrate in time using the third-order SSP Runge-Kutta method \cite{Gottlieb2001}.
The time step is chosen according to \eqref{eq:cfl-low}, with a CFL constant of $1/2$, i.e.\ $\Delta t = \frac{1}{2}\min_i \frac{\mathtt{m}_i}{2 \hat{\mathtt{d}}_{ii}}$.

\subsection{Convergence tests}

We first study the accuracy of the method on a simple 1D test case for the linear advection equation
\begin{equation}
  u_t + u_x = 0.
\end{equation}
The domain is $\Omega = [0,1]$, and the initial condition is $u_0(x) = \sin(2 \pi x)$.
Periodic boundary conditions are enforced.
To avoid temporal errors, the equations are integrated using an eighth order Runge-Kutta method, until a final time of $t=1$, at which point the exact solution coincides with the initial condition.
We use polynomial degrees $p = 0, 1, \ldots, 5$, on a sequence of uniformly refined meshes.
We compare the results both with and without the smoothness indicator (as described in Section \ref{sec:smoothness}).
The $L^1$ error is computed by comparing with the exact solution, and the results are shown in Table \ref{tab:convergence}.
Using the smoothness indicator, we observe the expected asymptotic rates for all cases except $p=1$ (for which case we do not expect the subcell resolution smoothness indicator to perform well.)
Without the smoothness indicator, the local maximum principle is strictly enforced.
As a consequence, only first-order convergence is observed for these cases, consistent with other results reported in the literature \cite{Khobalatte1994}.

\begin{table}
  \caption{
  Convergence results for smooth test case for the 1D advection equation, showing results with and without the smoothness indicator.}
  \label{tab:convergence}
  
  \centering
  \vspace{\floatsep}
  \setlength{\tabcolsep}{12pt}
  
\begin{tabular}{crlclc}
\toprule
& & \multicolumn{2}{c}{S.I. Enabled} & \multicolumn{2}{c}{S.I. Disabled} \\
\multicolumn{2}{r}{Elements} & \multicolumn{1}{c}{$L^1$ error} & Rate & \multicolumn{1}{c}{$L^1$ error} & Rate \\
\midrule
\multirow{4}{*}{$p=0$} & 8 & $6.02 \times 10^{-1}$ &  ---  & $6.02 \times 10^{-1}$ &  --- \\
& 16 & $4.54 \times 10^{-1}$ & 0.41 & $4.54 \times 10^{-1}$ & 0.41\\
& 32 & $2.93 \times 10^{-1}$ & 0.63 & $2.93 \times 10^{-1}$ & 0.63\\
& 64 & $1.69 \times 10^{-1}$ & 0.80 & $1.69 \times 10^{-1}$ & 0.80\\
\midrule
\multirow{4}{*}{$p=1$} & 8 & $2.80 \times 10^{-1}$ &  ---  & $2.80 \times 10^{-1}$ &  --- \\
& 16 & $1.08 \times 10^{-1}$ & 1.37 & $1.08 \times 10^{-1}$ & 1.37\\
& 32 & $4.95 \times 10^{-2}$ & 1.13 & $4.95 \times 10^{-2}$ & 1.13\\
& 64 & $2.47 \times 10^{-2}$ & 1.00 & $2.47 \times 10^{-2}$ & 1.00\\
\midrule
\multirow{4}{*}{$p=2$} & 8 & $5.28 \times 10^{-2}$ &  ---  & $6.55 \times 10^{-2}$ &  --- \\
& 16 & $2.33 \times 10^{-2}$ & 1.18 & $2.47 \times 10^{-2}$ & 1.41\\
& 32 & $9.05 \times 10^{-3}$ & 1.36 & $1.14 \times 10^{-2}$ & 1.12\\
& 64 & $2.56 \times 10^{-3}$ & 1.82 & $5.55 \times 10^{-3}$ & 1.03\\
\midrule
\multirow{4}{*}{$p=3$} & 8 & $2.31 \times 10^{-4}$ &  ---  & $1.64 \times 10^{-2}$ &  --- \\
& 16 & $1.14 \times 10^{-5}$ & 4.34 & $6.37 \times 10^{-3}$ & 1.37\\
& 32 & $7.10 \times 10^{-7}$ & 4.01 & $3.20 \times 10^{-3}$ & 0.99\\
& 64 & $4.43 \times 10^{-8}$ & 4.00 & $1.60 \times 10^{-3}$ & 1.00\\
\midrule
\multirow{4}{*}{$p=4$} & 8 & $6.48 \times 10^{-6}$ &  ---  & $1.15 \times 10^{-2}$ &  --- \\
& 16 & $2.05 \times 10^{-7}$ & 4.98 & $3.04 \times 10^{-3}$ & 1.92\\
& 32 & $6.46 \times 10^{-9}$ & 4.99 & $1.27 \times 10^{-3}$ & 1.26\\
& 64 & $2.02 \times 10^{-10}$ & 5.00 & $6.19 \times 10^{-4}$ & 1.04\\
\midrule
\multirow{4}{*}{$p=5$} & 8 & $2.06 \times 10^{-7}$ &  ---  & $7.82 \times 10^{-3}$ &  --- \\
& 16 & $3.19 \times 10^{-9}$ & 6.01 & $2.16 \times 10^{-3}$ & 1.85\\
& 32 & $4.98 \times 10^{-11}$ & 6.00 & $7.32 \times 10^{-4}$ & 1.56\\
& 64 & $8.59 \times 10^{-13}$ & 5.86 & $2.99 \times 10^{-4}$ & 1.29\\
\bottomrule
\end{tabular}
\end{table}

\subsection{1D Euler}

For a first set of initial test cases, we consider the one-dimensional Euler equations, given by
\begin{equation}
  \bm u_t + \bm F_x = 0,
\end{equation}
where $\bm u = (\rho, \rho v, \rho E)^\tr$, $\bm F = (\rho v, \rho v^2 + p, v(\rho E + p))^\tr$.
The pressure is given by the equation of state $p = (\gamma-1)\rho(E - v^2/2)$.

\subsubsection{Sod shock tube}

We first consider the classical Sod shock tube problem \cite{Sod1978}.
The domain is taken to be $\Omega = [-1/2, 1/2]$, and the initial conditions in primitive variables are given by
\begin{equation}
  \bm u_0(x) = \begin{cases}
    \bm u_L, \qquad x < 0, \\
    \bm u_R, \qquad x \geq 0,
    \end{cases}
  \qquad
  \bm u_L =
  \left( \begin{array}{c} \rho_L \\ v_L \\ p_L \end{array} \right)
  = \left( \begin{array}{c} 1 \\ 0 \\ 1 \end{array} \right),
  \qquad
  \bm u_R =
  \left( \begin{array}{c} \rho_R \\ v_R \\ p_R \end{array} \right)
  = \left( \begin{array}{c} 1/8 \\ 0 \\ 1/10 \end{array} \right).
\end{equation}
This problem gives rise to a rarefaction wave, a contact discontinuity, and a shock.
We integrate in time until $t=0.18$ with the number of degrees of freedom set to 256, using polynomial degrees $p=0,1,3,7$, so that the highest degree run is performed on a mesh with 32 elements.
Figure \ref{fig:sod} compares the final density and pressure with the exact solution to the Riemann problem.
From Figure \ref{fig:sod}, we observe that increasing the polynomial degree while simultaneously coarsening the mesh leads to somewhat better resolution of discontinuities in the solution, in particular at the contact discontinuity.
The solutions obtained using $p=3$ and $p=7$ are largely indistinguishable.

\begin{figure}
  \centering
  \includegraphics{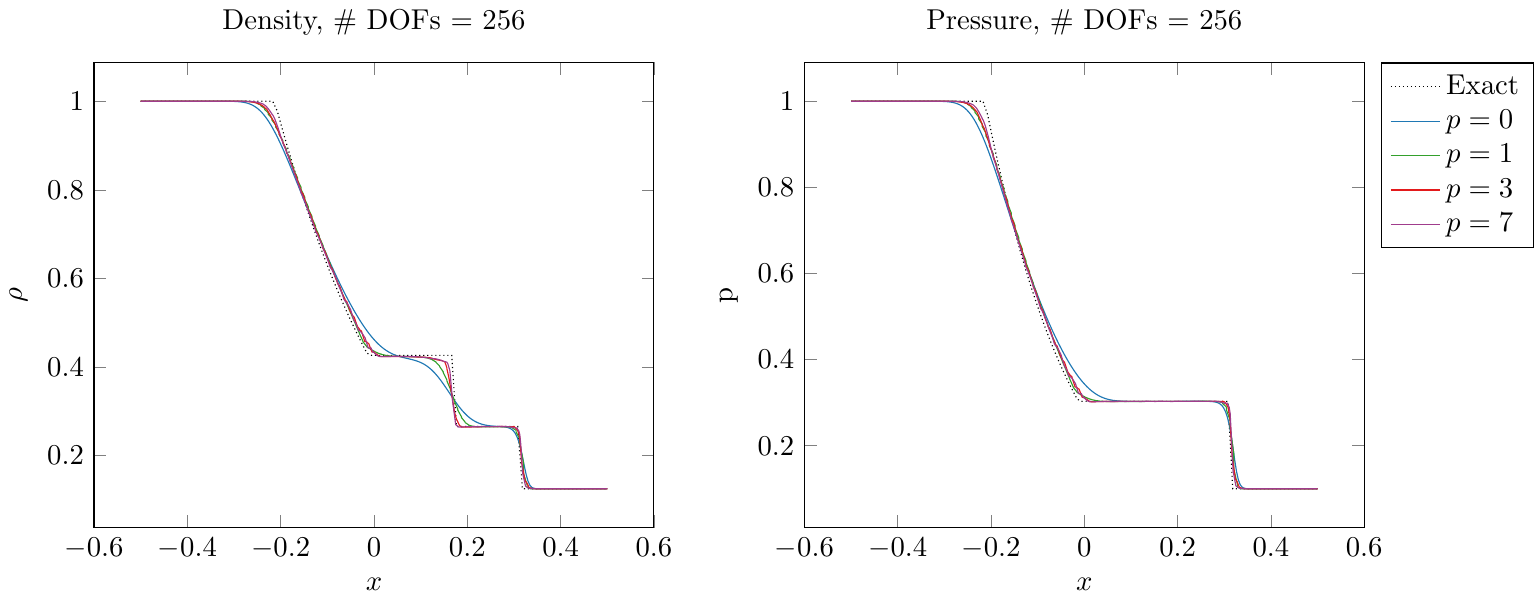}
  \caption{Density and pressure for the Sod shock tube problem at $t=0.18$.}
  \label{fig:sod}
\end{figure}

\subsubsection{Sine-shock interaction}

We now consider the Shu-Osher sine-shock interaction problem \cite{Shu1989}.
The domain is $\Omega = [-5, 5]$ and the initial condition is given by
\begin{equation}
  \bm u_0(x) = \begin{cases}
    \bm u_L, &\quad x < -4, \\
    \bm u_R, &\quad x \geq -4,
    \end{cases}
  \quad
  \left( \begin{array}{c} \rho_L \\ v_L \\ p_L \end{array} \right)
  = \left( \begin{array}{c} 3.857143 \\ 2.629369 \\ 10.3333 \end{array} \right),
  \quad
  \left( \begin{array}{c} \rho_R \\ v_R \\ p_R \end{array} \right)
  = \left( \begin{array}{c} 1 + 0.2\sin(5x) \\ 0 \\ 1 \end{array} \right).
\end{equation}
This test case can be challenging for shock-capturing schemes because the solution contains strong and weak shocks, as well as smooth oscillatory structures.
We integrate in time until $t=1.8$ with the number of degrees of freedom set to 512, using polynomial degrees $p=0,1,3,7$.
The $p=7$ run is performed on a mesh with 64 elements.
The final density and pressure are shown in Figure \ref{fig:sine-shock}.
These solutions are compared with a reference solution computed using $p=0$ with a fine mesh of 20{,}000 elements.
The strong shock is resolved well with all of the polynomial degrees, however the smooth structures are better resolved by using higher degree polynomials.
Some peak clipping is observed on the post-shock oscillations.

\begin{figure}
  \centering
  \includegraphics{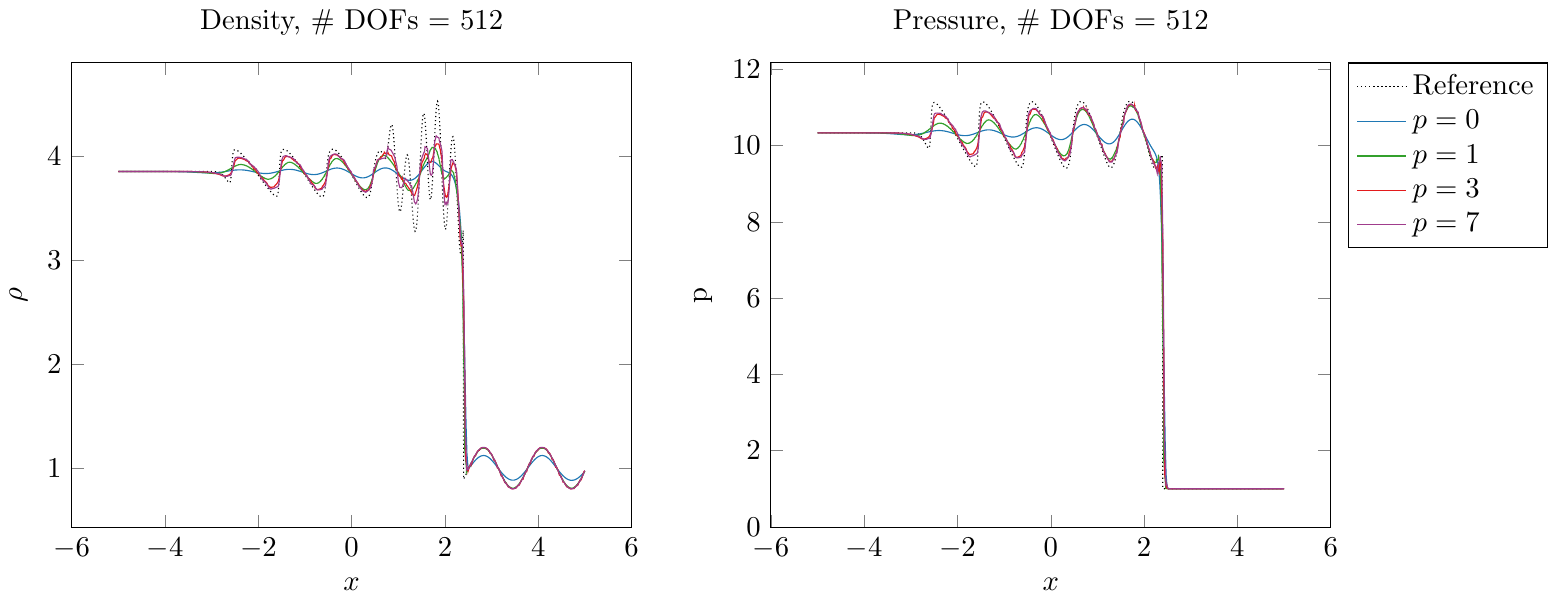}
  \caption{Density and pressure for the sine-shock interaction problem at $t=1.8$.}
  \label{fig:sine-shock}
\end{figure}

\subsection{Buckley-Leverett equation}

The Buckley-Leverett equation is a hyperbolic conservation law with non-convex flux function used to model porous media flow, defined by
\begin{equation}
  \frac{\partial u}{\partial t}
  + \frac{\partial}{\partial x} \left(
    \frac{4u^2}{4u^2 + (1-u)^2}
  \right) = 0.
\end{equation}
We consider the Riemann problem
\[
  u_0(x) = \begin{cases}
    -3, &\quad x < 0, \\
    3,  &\quad x \geq 0.
  \end{cases}
\]
We fix the number of degrees of freedom to be 256, and integrate in time until $t=0.25$ using polynomial degrees $p=0,1,3,7$.
The solution is shown in Figure \ref{fig:buckley}.
For this test case, the solutions obtained using the high-order flux-limited DG method compare well to the reference solution computed with $p=0$ on a mesh with 10{,}000 elements.

\begin{figure}
  \centering
  \includegraphics{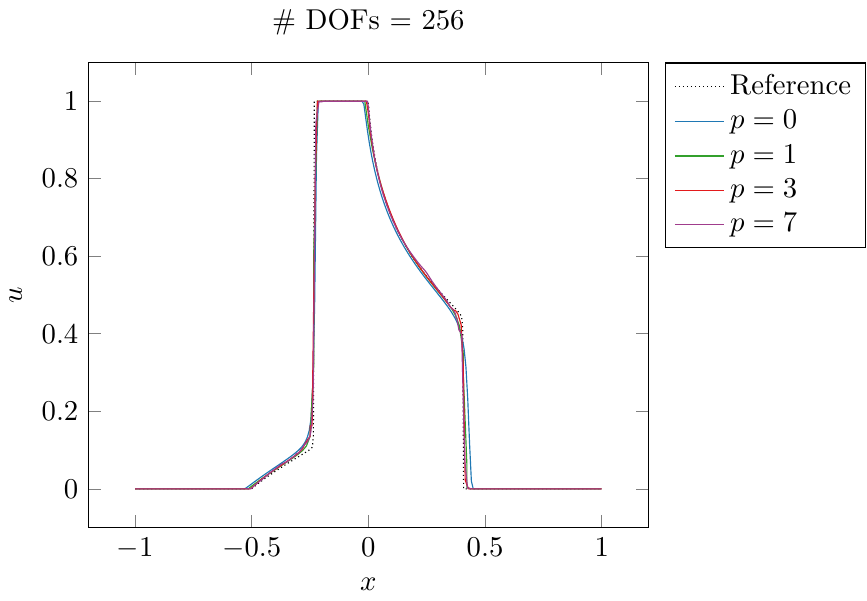}
  \caption{Riemann problem for the Buckley-Leverett equation at $t=0.25$ using $p=0,1,3,7$ with 256 degrees of freedom.}
  \label{fig:buckley}
\end{figure}

\subsection{2D linear advection: solid body rotation}
\label{sec:solid-body}

We consider the solid body rotation test, first proposed by LeVeque, which has since become a standard benchmark test case to assess the resolution of both smooth and discontinuous features \cite{LeVeque1996}.
The governing equation is the two-dimensional linear advection equation
\begin{equation}
  \frac{\partial u}{\partial t} + \nabla \cdot \left( \bm v u \right) = 0,
\end{equation}
with velocity field $\bm v = ( 2\pi(1/2 -y), 2 \pi (x - 1/2))^\tr$ in the domain $\Omega = [0,1] \times [0,1]$.
The initial conditions consist of a smooth bump, a cone, and a slotted cylinder.
Each of these bodies is defined in a disk of radius $r_0 = 0.15$, centered at some point $(x_0, y_0) \in \Omega$.
Let $r(x,y) = \frac{1}{r_0}\sqrtb{(x-x_0)^2 + (y-y_0)^2}$ denote the normalized distance to the center point.
The smooth bump is centered at $(x_0, y_0) = (0.25, 0.5)$, and is defined by
\[
  u_{\rm bump}(x,y) = \frac{1 + \cos(\pi r(x,y))}{4}.
\]
The cone is centered at $(x_0, y_0) = (0.5, 0.25)$, and is defined by
\[
  u_{\rm cone}(x,y) = 1 - r(x,y).
\]
The slotted cylinder is centered at $(x_0, y_0) = (0.5, 0.75)$, and is defined by
\[
  u_{\rm cyl}(x,y) = \begin{cases}
    1, &\quad\text{if $|x - x_0| \geq 0.025$ or $y >= 0.85$,}\\
    0, &\quad\text{otherwise.}
  \end{cases}
\]
The initial condition $u_0$ is defined using the above functions on the each of the three disks, and is set to zero elsewhere.
We integrate in time until $t=1$, at which point a full revolution has completed.
In Figure \ref{fig:solid-body-comparison}, we show the solution obtain on a $64 \times 64$ Cartesian grid using the high-order DG method \eqref{eq:dg-strong}, the low-order invariant domain preserving method \eqref{eq:dg-idp}, and the flux-corrected method \eqref{eq:fct-update} with subcell limiting.
The standard (unlimited) DG method results in clear oscillations and overshoots and undershoots, in particular around the slotted cylinder.
This is evident from the minimum and maximum values after one revolution, which are -0.21 and 1.16, respectively.
The low-order IDP method is bounds preserving, but clearly very dissipative.
The flux-corrected method, obtained by performing a subcell bounds-preserving blending of the low-order IDP method and the high-order target method, results in a solution without oscillations or new extrema.
The $L^1$ accuracy of the flux-corrected method is comparable to that of the high-order method for this problem.

\begin{figure}
  \centering

  \begin{minipage}[b]{0.3\linewidth}
     \centering
     High-Order DG\\[12pt]
     \includegraphics[width=\linewidth]{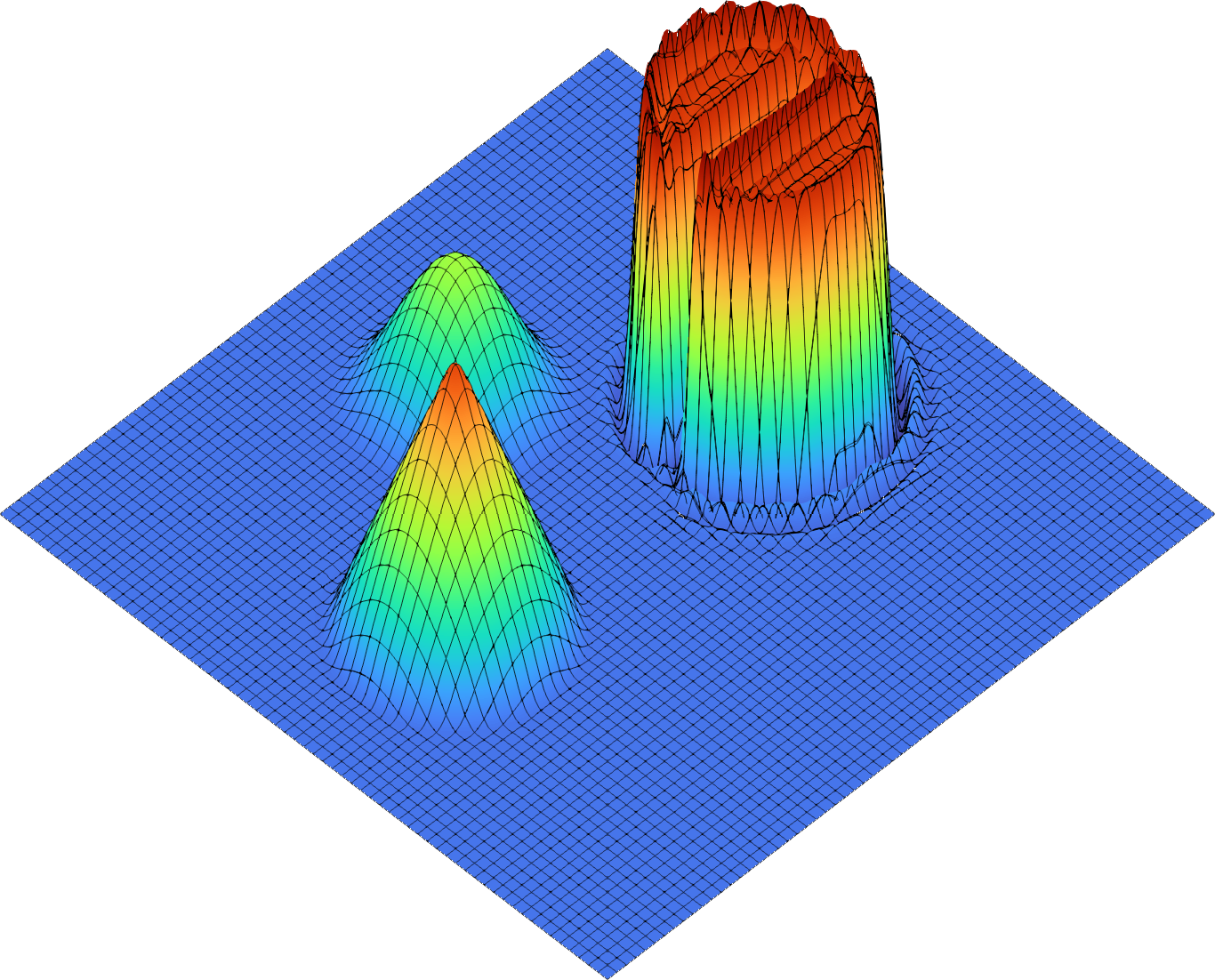}\\[12pt]
     \pbox{\linewidth}{ \small
     $[\min u_h, \max u_h] = [-0.21, 1.16]$\\
     $\|u - u_h\|_{L^1} = 1.0\times10^{-2}$}
  \end{minipage}
  \hspace*{\fill}
  \begin{minipage}[b]{0.3\linewidth}
     \centering
     Low-Order IDP\\[12pt]
     \includegraphics[width=\linewidth]{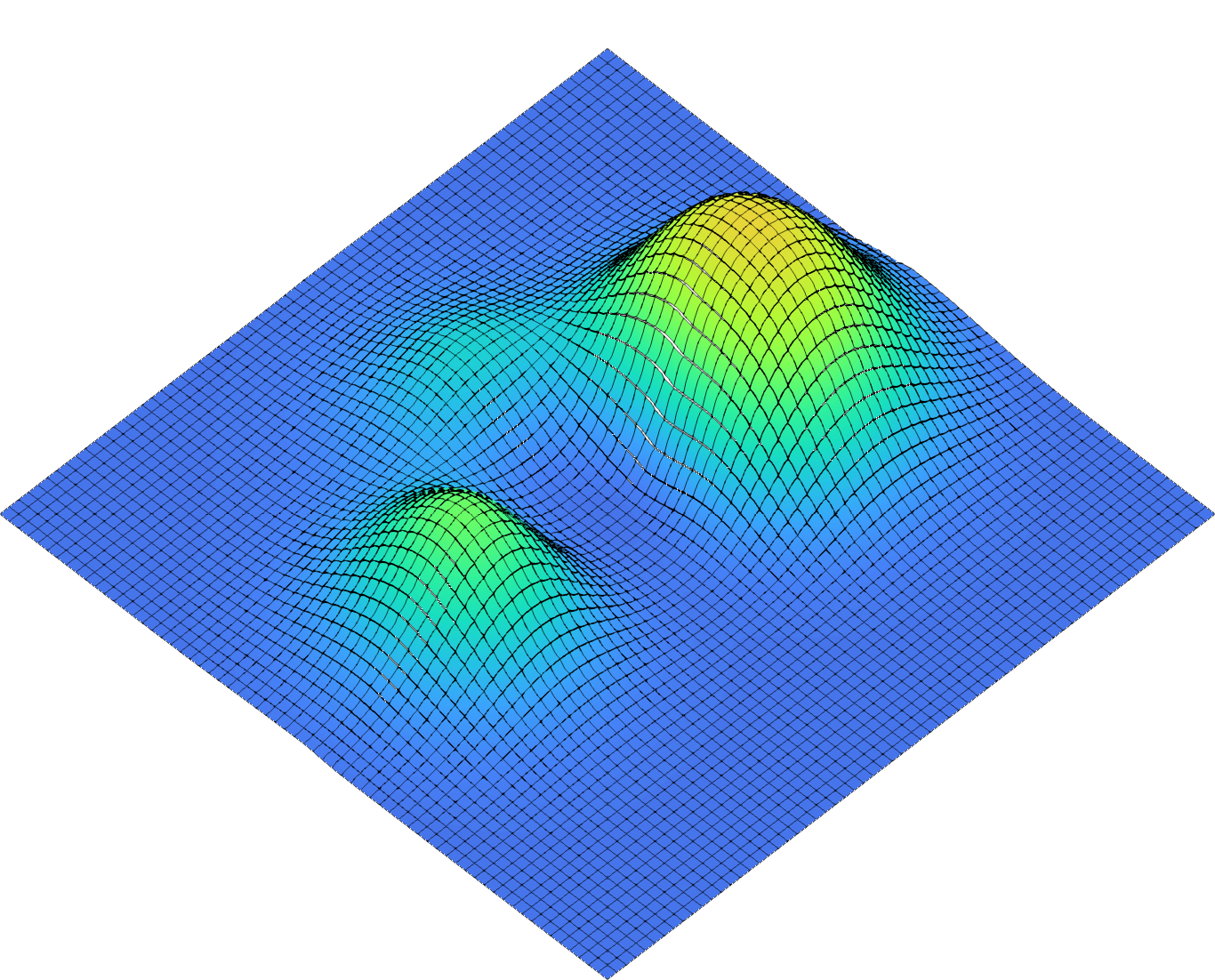}\\[12pt]
     \pbox{\linewidth}{ \small
     $[\min u_h, \max u_h] = [0, 0.66]$\\
     $\|u - u_h\|_{L^1} = 8.5\times10^{-2}$}
  \end{minipage}
  \hspace*{\fill}
  \begin{minipage}[b]{0.3\linewidth}
     \centering
     FCT\\[12pt]
     \includegraphics[width=\linewidth]{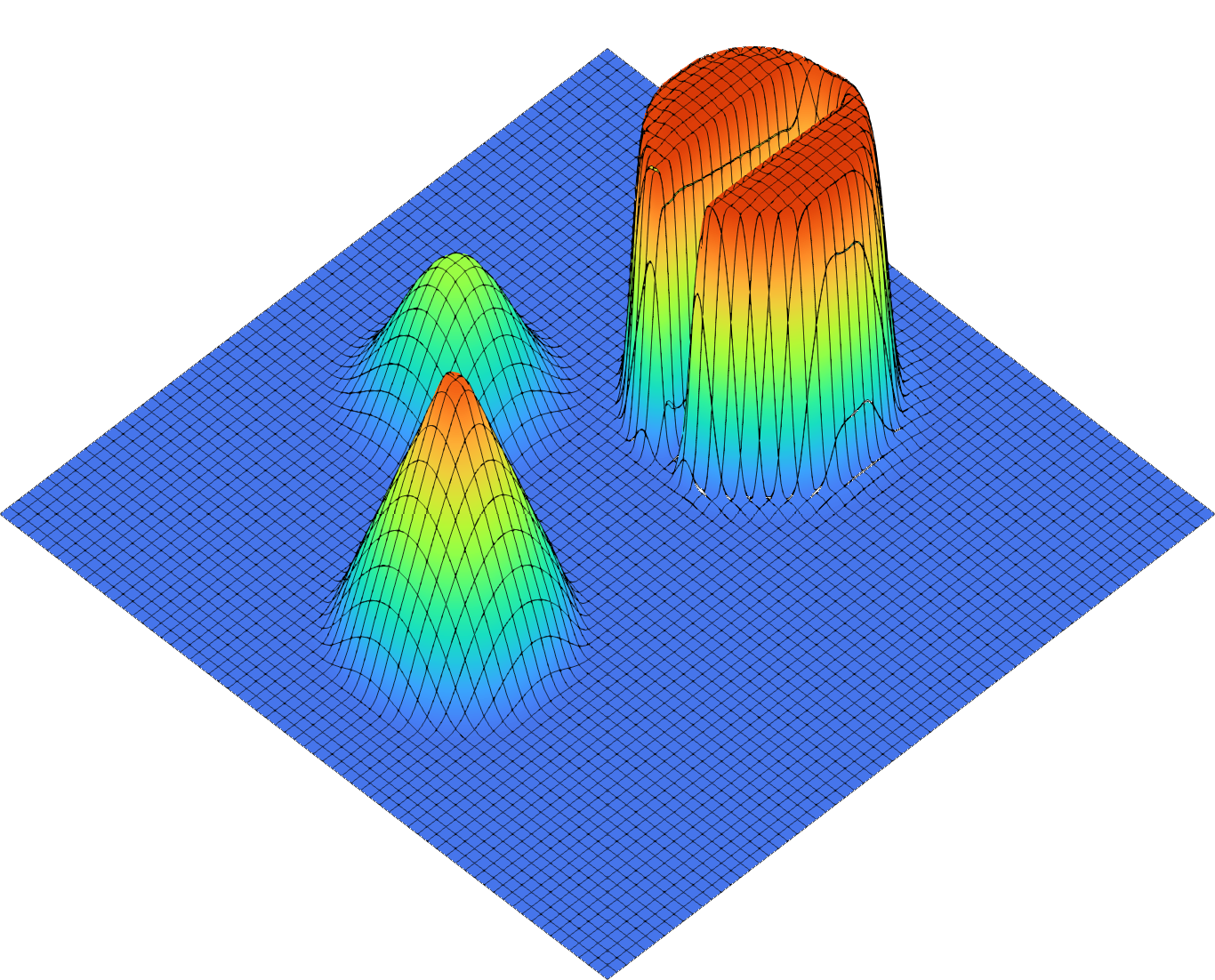}\\[12pt]
     \pbox{\linewidth}{ \small
     $[\min u_h, \max u_h] = [0, 1]$\\
     $\|u - u_h\|_{L^1} = 1.1\times10^{-2}$}
  \end{minipage}

  \caption{Solid body rotation test for the 2D linear advection equation on a $64 \times 64$ Cartesian grid with $p=3$.
    Left panel: high-order DG-SEM with no limiting.
    Center panel: low-order sparsified IDP method.
    Right panel: bounds-preserving subcell limiter.}
  \label{fig:solid-body-comparison}
\end{figure}

\subsection{2D Burgers equation}

Consider the two-dimensional Burgers equation
\begin{equation}
  \frac{\partial u}{\partial t} + \nabla \cdot \left( \frac{1}{2} u^2 \bm v \right) = 0,
\end{equation}
with constant velocity vector $\bm v = (1,1)^\tr$ in the domain $\Omega = [0,1] \times [0,1]$.
We consider the piecewise constant initial condition
\begin{equation}
  u_0(x,y) = \begin{cases}
    -1   \quad& \text{if $x > 0.5$ and $y > 0.5$,}\\
    -0.2 \quad& \text{if $x \leq 0.5$ and $y > 0.5$,}\\
     0.5 \quad& \text{if $x \leq 0.5$ and $y \leq 0.5$,}\\
     0.8 \quad& \text{if $x > 0.5$ and $y \leq 0.5$.}
  \end{cases}
\end{equation}
This problem was considered in \cite{Shu1989,Guermond2011}.
The exact solution (determined analytically, cf.~\cite{Wagner1983}) is imposed as boundary conditions, and the equations are integrated until a final time of $t = 0.5$.
We use polynomial degree $p=2,5,11$ on a sequence of increasingly coarse Cartesian grids $(n_{\rm 1D} = 40, 20, 10)$, with total number of degrees of freedom equal to $120^2$.
The solution is shown in Figure \ref{fig:burgers2d}.
When compared with a reference solution computed on a fine mesh ($p=1, n_{\rm 1D} = 512$), the solution is well-resolved even on the coarse mesh.
Discontinuities in the solution are captured well even when not aligned with element boundaries.
\begin{figure}
  \begin{minipage}{0.8\linewidth}
    \def\burgersfigscale{0.3}

    \hspace*{\fill}
    \begin{minipage}{\burgersfigscale\linewidth}
      \centering
      \includegraphics[width=\linewidth]{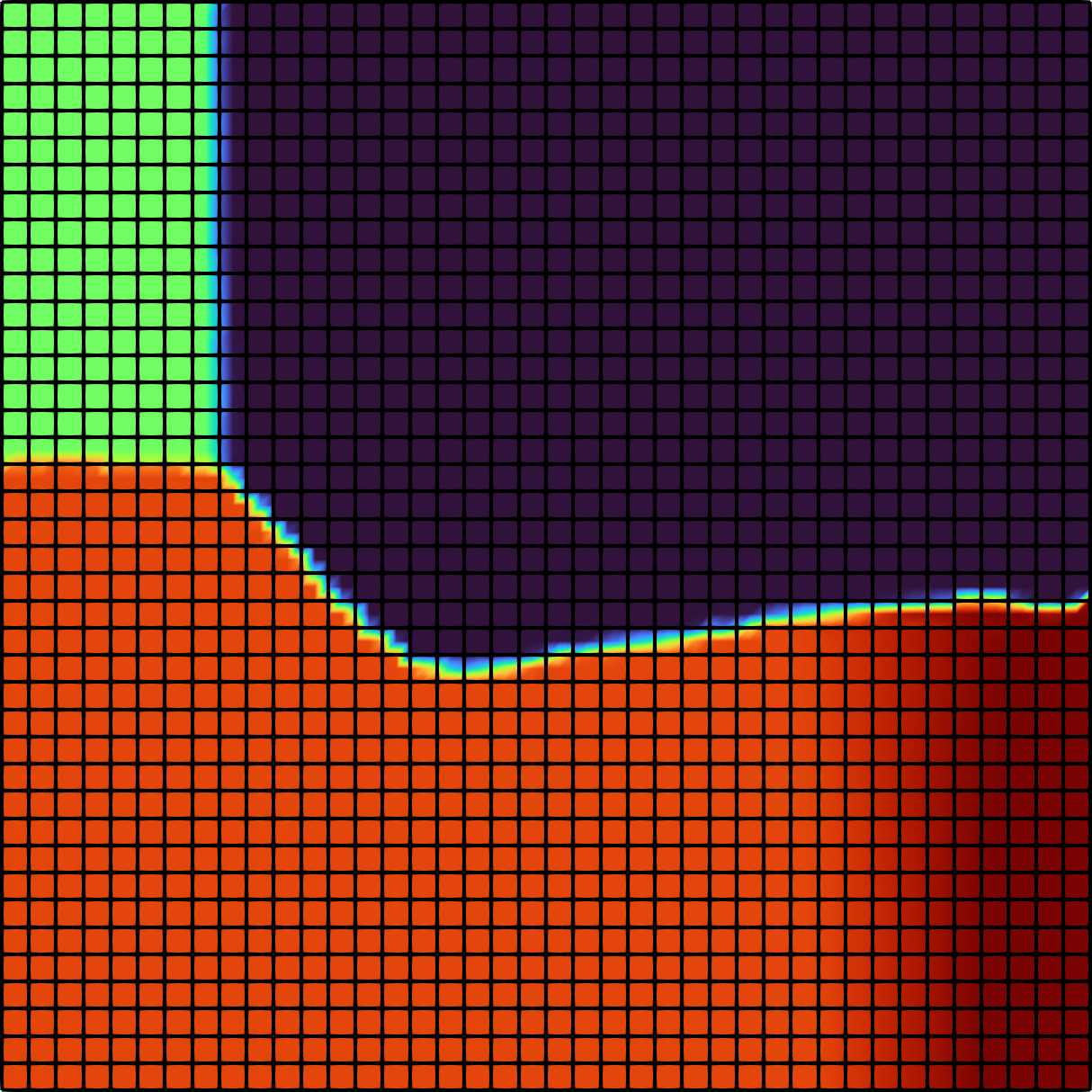}

      $p=2, n_{\rm 1D} = 40$
    \end{minipage}
    \hspace*{\fill}
    \begin{minipage}{\burgersfigscale\linewidth}
      \centering
      \includegraphics[width=\linewidth]{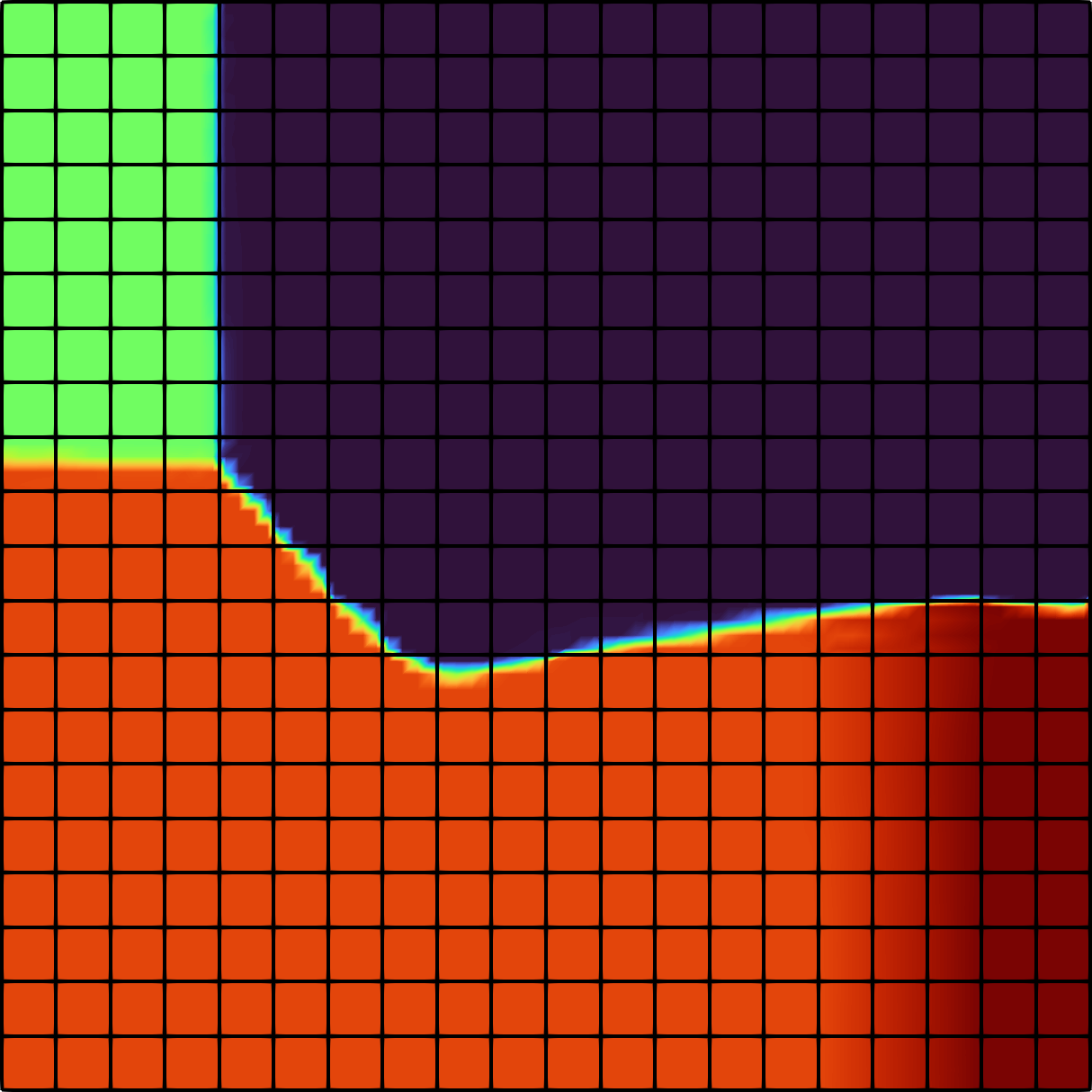}

      $p=5, n_{\rm 1D} = 20$
    \end{minipage}
    \hspace*{\fill}

    \vspace{\floatsep}

    \hspace*{\fill}
    \begin{minipage}[t]{\burgersfigscale\linewidth}
      \centering \vskip0pt
      \includegraphics[width=\linewidth]{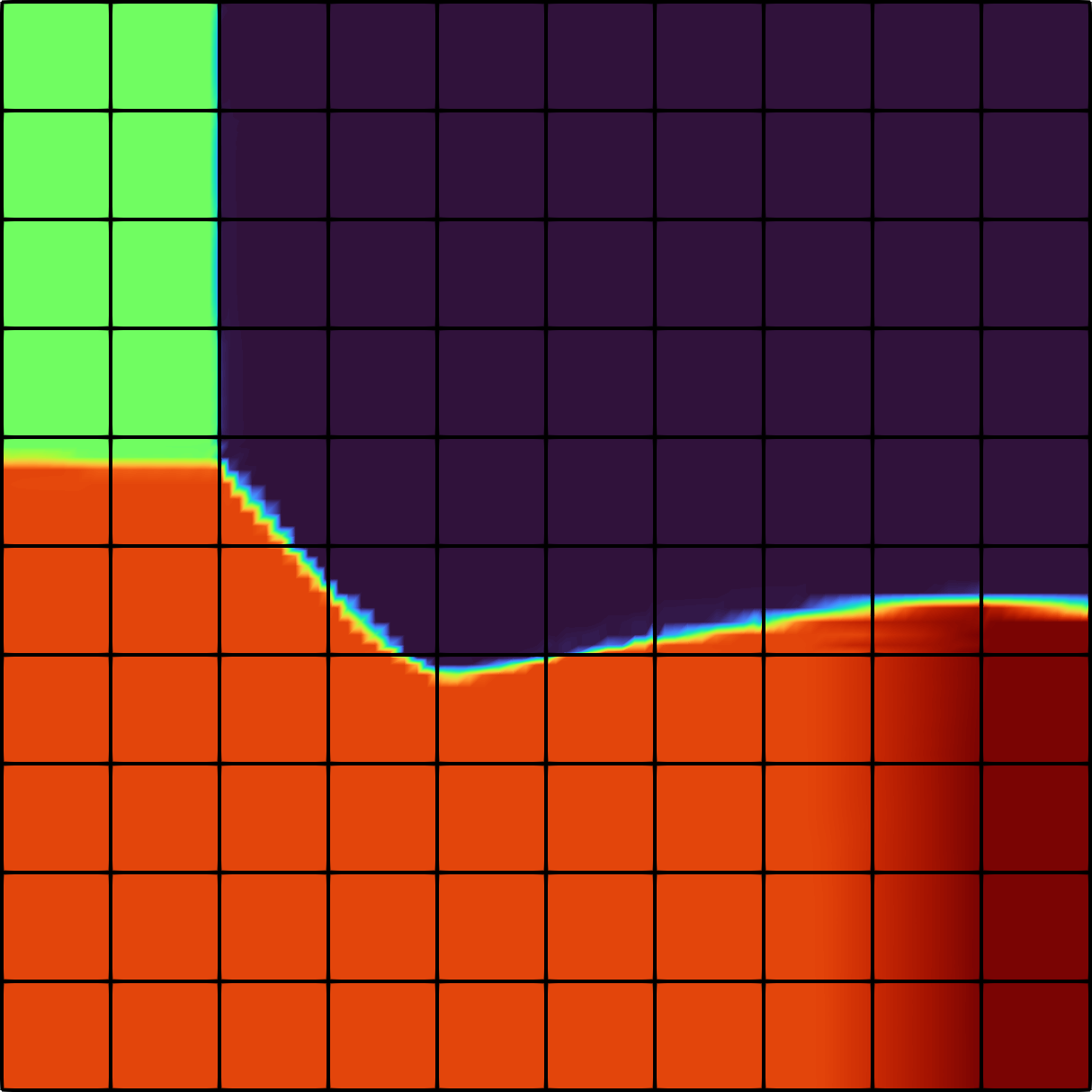}

      $p=11, n_{\rm 1D} = 10$
    \end{minipage}
    \hspace*{\fill}
    \begin{minipage}[t]{\burgersfigscale\linewidth}
      \centering \vskip0pt
      \includegraphics[width=\linewidth]{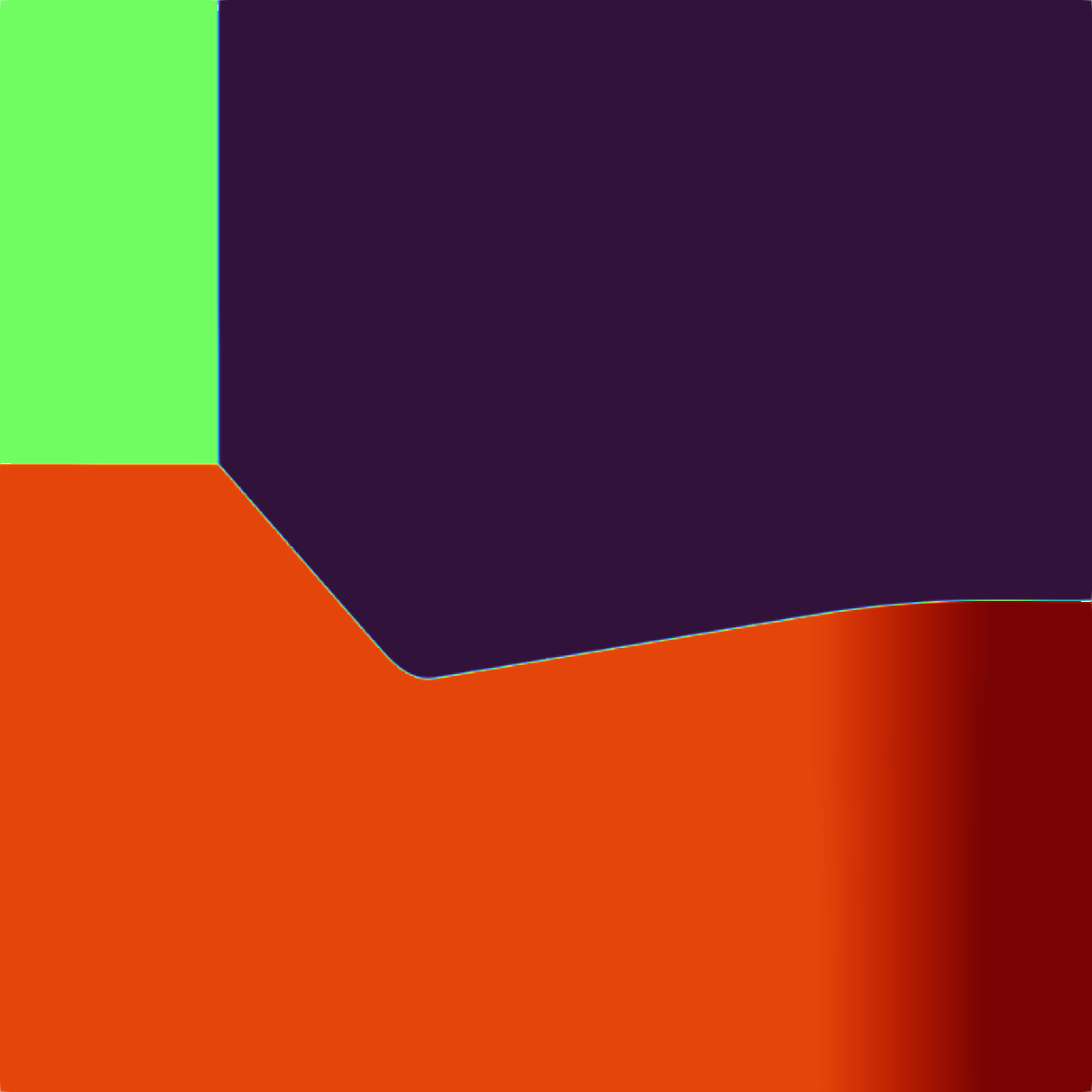}

      Reference ($p=1, n_{\rm 1D} = 512$)
    \end{minipage}
    \hspace*{\fill}
  \end{minipage}
  \begin{minipage}{0.19\linewidth}
    \includegraphics{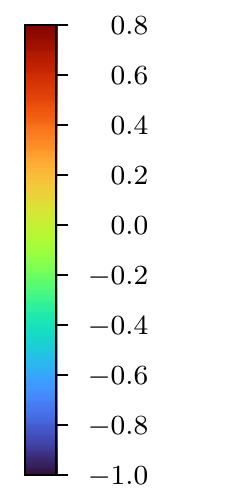}
  \end{minipage}

  \caption{2D Burgers equation Riemann problem, using $p=2,5,11$ with fixed number of degrees of freedom.}
  \label{fig:burgers2d}
\end{figure}

\subsection{2D Euler Riemann problem}

We now test the method on a Riemann problem for the 2D Euler equations (``configuration 12''), often used as a benchmark problem \cite{Liska2003,Schulz-Rinne1993,Lax1998}.
The spatial domain is taken to be $\Omega = [0,1] \times [0,1]$, and the initial conditions are defined by piecewise-constant data on each of the quadrants,
\begin{equation}
\left\{
\begin{aligned}
  \rho &= 4/5, \quad & \bm v &= (0,0), & p &= 1,
    \qquad && 0 < x < 1/2, \quad 0 < y < 1/2,\\
  \rho &= 1,   \quad & \bm v &= (3/\sqrt{17},0), & p &= 1,
    \qquad && 0 < x < 1/2, \quad 1/2 < y < 1, \\
  \rho &= 1,   \quad & \bm v &= (0,3/\sqrt{17}), & p &= 1,
    \qquad && 1/2 < x < 1, \quad 0 < y < 1/2, \\
  \rho &= 17/32,   \quad & \bm v &= (0,0), & p &= 2/5,
    \qquad && 1/2 < x < 1, \quad 1/2 < y < 1.
\end{aligned}
\right.
\end{equation}
The problem is made periodic on the enlarged domain $[0,2] \times [0,2]$ by reflecting the initial conditions about the point $(1,1)$, as described in \cite{Guermond2011}.
The solution is taken to be the restriction of the periodic solution to the subdomain $[0,1] \times [0,1]$.
The equations are integrated until time $t = 0.25$.
Polynomial degrees $p=1$ and $p=3$ are used, on $128 \times 128$ and $64 \times 64$ Caetesian grids, respectively, such that the total number of degrees of freedom is fixed for both calculations.
The density and pressure fields of the final solutions are shown in Figure \ref{fig:riemann2d}.
Both solutions resolve the large-scale features, including the shocks and contact discontinuities.
Although the number of degrees of freedom is the same for both cases, the solution obtained using $p=3$ polynomials shows sharper interfaces, and better-resolved small-scale features.

\begin{figure}
  \def\riemannfigscale{0.35}
  \hspace*{\fill}
  \begin{minipage}{\riemannfigscale\linewidth}
    \centering
    \includegraphics[width=\linewidth]{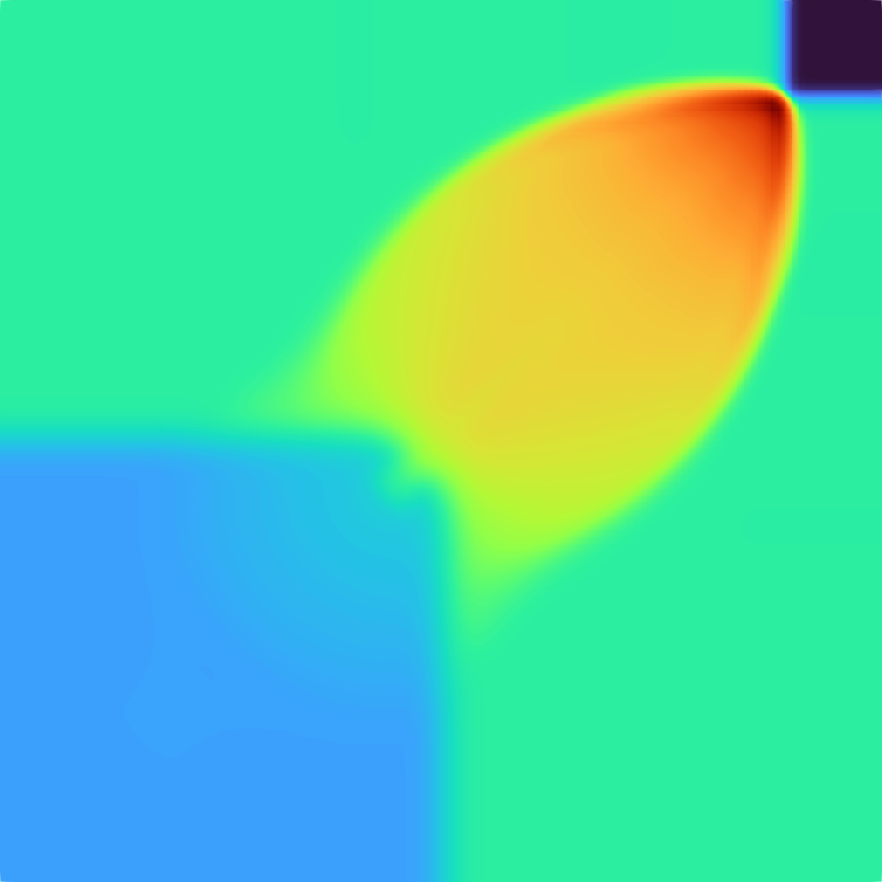}

    Density, \quad $p=1, n_{\rm 1D} = 128$
  \end{minipage}
  \hspace*{\fill}
  \begin{minipage}{\riemannfigscale\linewidth}
    \centering
    \includegraphics[width=\linewidth]{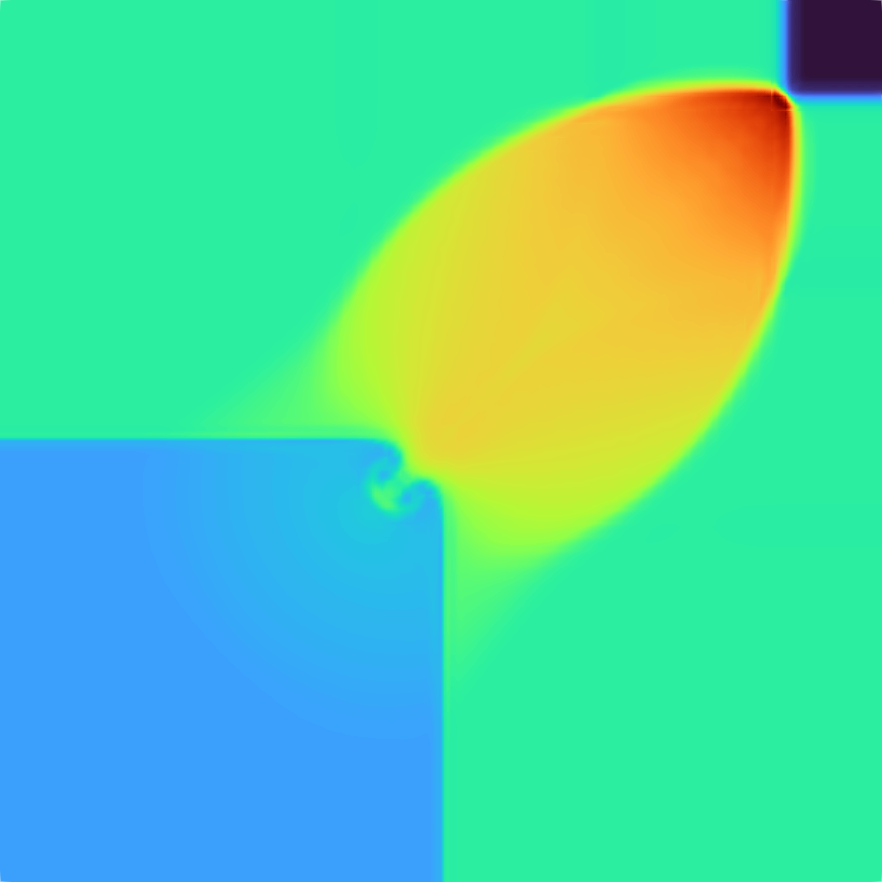}

    Density, \quad $p=3, n_{\rm 1D} = 64$
  \end{minipage}
  \hspace*{\fill}

  \vspace{\floatsep}

  \hspace*{\fill}
  \begin{minipage}{\riemannfigscale\linewidth}
    \centering
    \includegraphics[width=\linewidth]{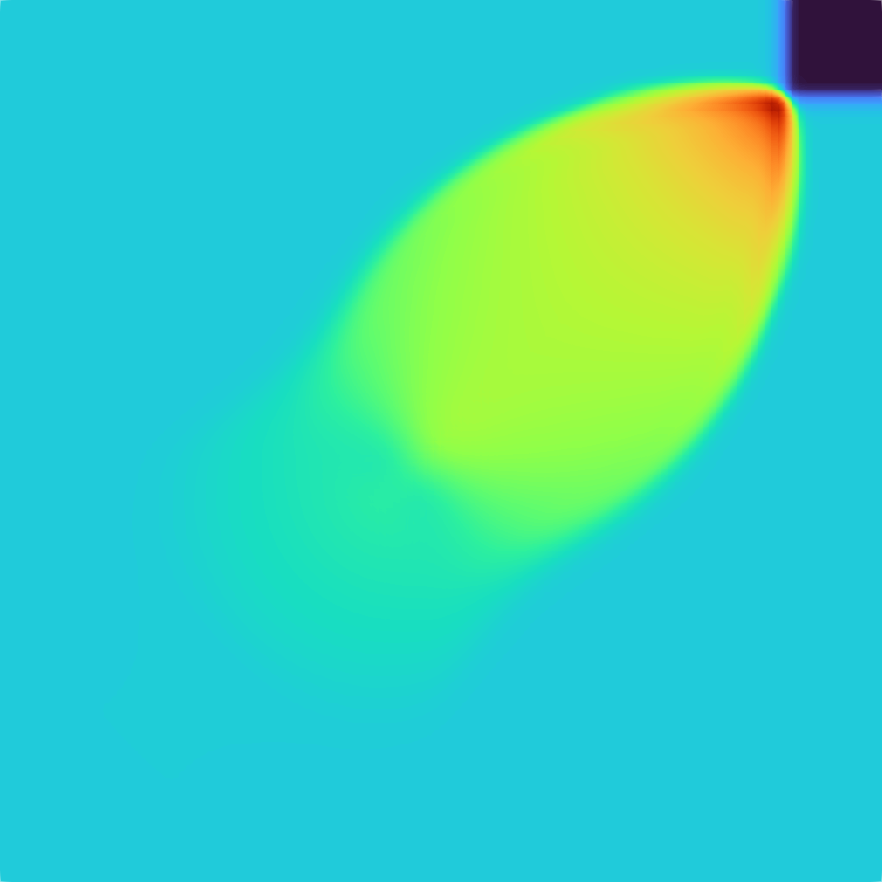}

    Pressure, \quad $p=1, n_{\rm 1D} = 128$
  \end{minipage}
  \hspace*{\fill}
  \begin{minipage}{\riemannfigscale\linewidth}
    \centering
    \includegraphics[width=\linewidth]{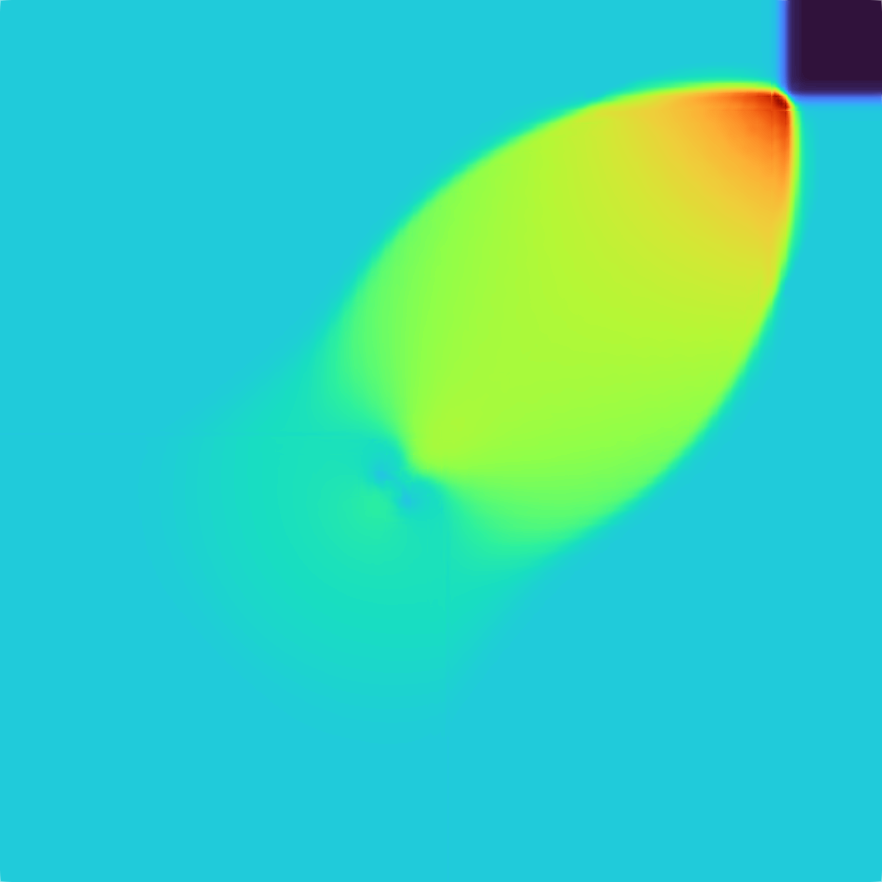}

    Pressure, \quad $p=3, n_{\rm 1D} = 64$
  \end{minipage}
  \hspace*{\fill}

  \caption{Density and pressure for 2D Euler equation Riemann problem, using $p=1,3$ with fixed number of degrees of freedom.}
  \label{fig:riemann2d}
\end{figure}

\subsection{Double Mach reflection}

Finally, we consider the double Mach reflection case of Woodward and Colella \cite{Woodward1984}.
This test cases consists of an incoming Mach 10 shock, that makes a 60$^\circ$ angle with a reflecting wall.
The undisturbed state ahead of the shock has density $\rho = 1.4$ and pressure $p=1$.
The problem is modeled in the rectangular domain $[0,4] \times [0,1]$, such that the bottom boundary (beginning at $x=1/6$) represents the inclined wedge.
The left (inflow) boundary and the interval $[0,1/6]$ on the bottom boundary are assigned the post-shock state.
The interval $[1/6,4]$ on the bottom boundary is assigned slip boundary conditions, and the top boundary is assigned a prescribed state using the exact shock speed.
Outflow conditions are enforced at the right boundary.
A fine mesh with $2400 \times 600$ elements with $p=3$ is used.
The density field and contours are shown in Figure \ref{fig:dmr}.
Small features such as the Kelvin-Helmholtz instability shown in the zoom-ins are indicative of the low dissipation of the scheme.

\begin{figure}
\centering
\includegraphics[width=\linewidth]{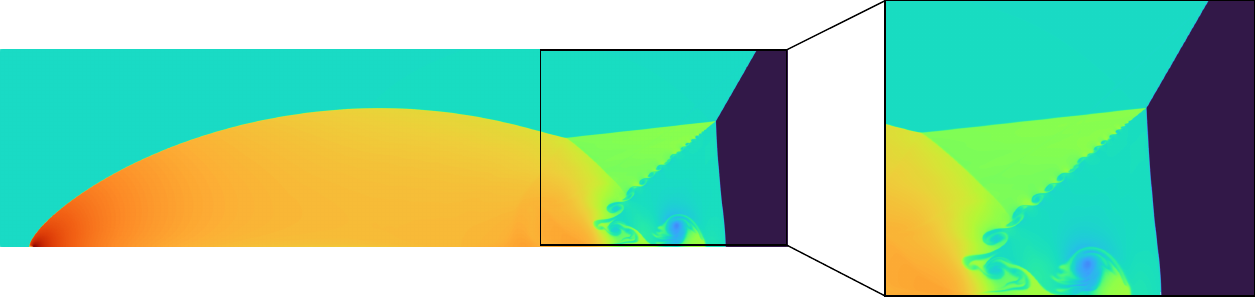}

\vspace{\floatsep}

\includegraphics[width=\linewidth]{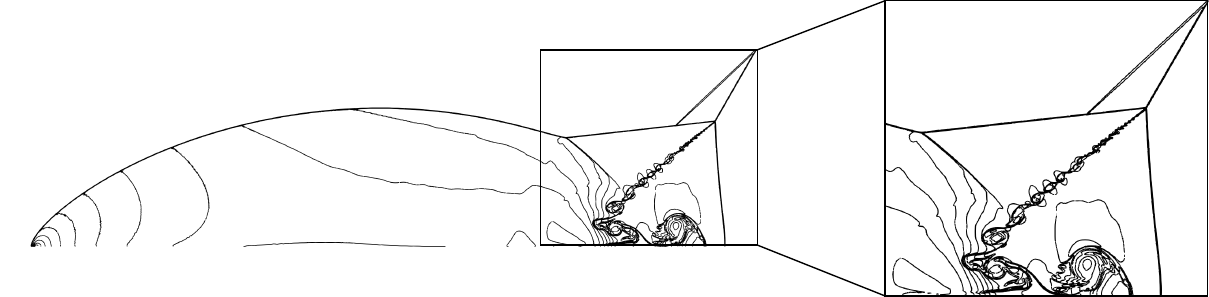}
\caption{Density for double Mach reflection problem at $t=0.275$, showing zoom-in on three-shock interaction point. Bottom panels: 30 equispaced density contours.}
\label{fig:dmr}
\end{figure}

\section{Conclusions}
\label{sec:conclusions}

In this work, we have presented a discontinuous Galerkin spectral element method with convex limiting for hyperbolic conservation laws.
This method preserves any specified set of invariant domain properties (e.g.\ local maximum principles, positivity of pressure and density, minimum principle for specific entropy, etc.).
The method is based on an efficient dimension-by-dimension subcell blending of the target high-order (unlimited) DG-SEM method, and a low-order, invariant domain preserving (IDP), sparsified scheme based on a graph viscosity approach.
Notably, the quality of this low-order IDP method does not degrade as the polynomial degree of the target method is increased, in contrast to non-sparsified graph viscosity approaches.
As a result, improved solution quality is obtained by using higher order target schemes on a variety of benchmark problems.
Additionally, a subcell resolution smoothness indicator is shown to be effective at reducing the peak clipping effect at smooth extrema.

\section{Acknowledgements}

The author acknowledges H.~Hajduk and D.~Kuzmin for insightful conversations and comments on this work.
\newline
\newline
{\small
This work was performed under the auspices of the U.S. Department of Energy by Lawrence Livermore National Laboratory under Contract DE-AC52-07NA27344.
LLNL-JRNL-808645.
This document was prepared as an account of work sponsored by an agency of the United States government.
Neither the United States government nor Lawrence Livermore National Security, LLC, nor any of their employees makes any warranty, expressed or implied, or assumes any legal liability or responsibility for the accuracy, completeness, or usefulness of any information, apparatus, product, or process disclosed, or represents that its use would not infringe privately owned rights.
Reference herein to any specific commercial product, process, or service by trade name, trademark, manufacturer, or otherwise does not necessarily constitute or imply its endorsement, recommendation, or favoring by the United States government or Lawrence Livermore National Security, LLC.
The views and opinions of authors expressed herein do not necessarily state or reflect those of the United States government or Lawrence Livermore National Security, LLC, and shall not be used for advertising or product endorsement purposes.
}

\bibliographystyle{siamplain}
\bibliography{refs}

\end{document}